\documentclass[preprint,reqno]{imsart}

\RequirePackage[OT1]{fontenc}
\RequirePackage{amsthm,amsmath}
\RequirePackage[numbers]{natbib}
\RequirePackage[colorlinks,citecolor=blue,urlcolor=blue]{hyperref}

\arxiv{arXiv:1701.07633}
\usepackage{bbm}
\usepackage{mathtools}
\usepackage{amsmath}
\usepackage[title,page]{appendix}
\startlocaldefs
\numberwithin{equation}{section}
\theoremstyle{plain}

\endlocaldefs
\newtheorem{theorem}{Theorem}[section]
\newtheorem{lemma}[theorem]{Lemma}
\newtheorem{remark}[theorem]{Remark}

\newtheorem{proposition}[theorem]{Proposition}

\allowdisplaybreaks[1]
\begin{document}

\begin{frontmatter}
\title{Diffusion approximations via Stein's method and time changes}
\runtitle{Diffusion approximations via Stein's method}

\begin{aug}
\author{\fnms{Miko{\l}aj J.} \snm{Kasprzak}\ead[label=e1]{kasprzak@stats.ox.ac.uk}},

\runauthor{Miko{\l}aj J. Kasprzak}

\affiliation{University of Oxford}

\address{Miko{\l}aj Kasprzak\\
Department of Statistics\\
University of Oxford\\
24-29 St Giles'\\
Oxford, OX1 3LB\\
United Kingdom\\
\printead{e1}\\
\phantom{E-mail:\ }}
\end{aug}

\begin{abstract}
We extend the ideas of \cite{diffusion} and use Stein's method to obtain a bound on the distance between a scaled time-changed random walk and a time-changed Brownian Motion. We then apply this result to bound the distance between a time-changed compensated scaled Poisson process and a time-changed Brownian Motion. 

This allows us to bound the distance between a process whose dynamics resemble those of the Moran model with mutation and a process whose dynamics resemble those of the Wright-Fisher diffusion with mutation upon noting that the former may be expressed as a difference of two time-changed Poisson processes and the diffusive part of the latter may be expressed as a time-changed Brownian Motion.

The method is applicable to a much wider class of examples satisfying the Stroock-Varadhan theory of diffusion approximation (\cite{stroock}). 
\end{abstract}



\end{frontmatter}
\section{Introduction}
In his seminal paper \cite{stein}, Charles Stein introduced a method for proving normal approximations and obtained a bound on the speed of convergence to the standard normal distribution. He observed that a random variable $Z$ has the standard normal law if and only if $\mathbbm{E}Zf(Z)=\mathbbm{E}f'(Z)$ for all smooth functions $f$. Therefore, if, for a random variable $W$ with mean $0$ and variance $1$, $\mathbbm{E}f'(W)-\mathbbm{E}Wf(W)$ is close to zero for a large class of functions $f$, then the law of $W$ should be approximately Gaussian. He then proposed that, instead of evaluating $|\mathbbm{E}h(W)-\mathbbm{E}h(Z)|$ directly for a given function $h$, one can first find an $f=f_h$ solving the following \textit{Stein equation}:
$$f'(w)-wf(w)=h(w)-\mathbbm{E}h(Z)$$
and then find a bound on $|\mathbbm{E}f'(W)-\mathbbm{E}Wf(W)|$. This approach often turns out to be much easier, due to some bounds on the solutions $f_h$, which can be derived in terms of the derivatives of $h$. Since then, Stein's method has been significantly developed and extended to approximations by distributions other than normal.

The aim of Stein's method is to find a bound on the quantity $|\mathbbm{E}_{\nu_n}h-\mathbbm{E}_{\mu}h|$, where $\mu$ is the target (known) distribution, $\nu_n$ is the approximating law and $h$ is chosen from a suitable class of real-valued test functions $\mathcal{H}$. The idea is to find an operator $\mathcal{A}$ acting on a class of real-valued functions such that $$\left(\forall f\in\text{Domain}(\mathcal{A})\quad\mathbbm{E}_{\nu}\mathcal{A}f=0\right)\quad \Longleftrightarrow \quad\nu=\mu,$$
where $\mu$ is our target distribution. In the next step, for a given function $h\in\mathcal{H}$, a solution $f=f_h$ to the following Stein equation:
$$\mathcal{A}f=h-\mathbbm{E}_{\mu}h$$
is sought and its properties studied. Finally, using various mathematical tools (among which the most popular are Taylor's expansions in the continuous case, Malliavin calculus, as described in \cite{nourdin}, and coupling methods), a bound is sought for the quantity $|\mathbbm{E}_{\nu_n}\mathcal{A}f_h|$.

An accessible account of the method can be found, for example, in the surveys \cite{reinert} and \cite{ross} as well as the books \cite{janson} and \cite{normal_approx}, which treat the cases of Poisson and normal approximation, respectively, in detail. \cite{swan} is a database of information and publications connected to Stein's method.

Approximations by infinite-dimensional laws have not been covered in the Stein's method literature very widely, with the notable exceptions of \cite{diffusion},  \cite{functional_combinatorial} and recently \cite{Coutin}. We will focus on the ideas taken from \cite{diffusion}, which provides bounds on the Brownian Motion approximation of a one-dimensional scaled random walk and some other one-dimensional processes including scaled sums of locally dependent random variables and examples from combinatorics. In the sequel, we show that the approach presented in \cite{diffusion} can be extended to time-changes of Brownian Motion, including diffusions in the natural scale.

The most important example we apply our theory to is the approximation of the Moran model with mutation by the Wright-Fisher diffusion with mutation. The former, first introduced in \cite{moran} as an alternative for the Wright-Fisher model (first formally described in \cite{wright}) is one of the simplest and most important models of the genetic drift, i.e. the change in the frequencies of alleles in a population. It assumes that the population is divided into two allelic types (A and a) and the frequency of each of the alleles is governed by a birth-death process and an independent mutation process. Specifically, in a population of size $n$, at exponential rate ${n\choose 2}$, a pair of genes is sampled uniformly at random. Then one of them is selected at random to die and the other one gives birth to another gene of the same allelic type. In addition, every gene of type a changes its type independently at rate $\nu_2$ and every gene of type A changes its type independently at rate $\nu_1$. The model then looks at the proportion of a-genes in the population.

On the other hand, the Wright-Fisher model is a discrete Markov chain and does not allow for overlapping generations. Specifically, each step represents a generation. In generation $k$ each of the $n$ individuals chooses its parent independently, uniformly at random from the individuals present in generation $k-1$ and inherits its genetic type. This model also then looks at the proportion of a-genes in the population.

The Moran model turns out to be easier to study mathematically. It may be proved, for instance using the Stroock-Varadhan theory of diffusion approximation (see \cite{stroock}), that it converges weakly to the Wright-Fisher diffusion, which is also a scaling limit of the Wright-Fisher model. We show how to put a bound on the speed of this convergence. The Wright-Fisher diffusion is often used in practice in genetics for inference concerning large populations. It is given by the equation:
$$dM(t)=\gamma(M(t))dt+\sqrt{M(t)(1-M(t))}dB_t,$$
where $\gamma:[0,1]\to\mathbbm{R}$ encompasses mutation.  For a discussion of probabilistic models in genetics see \cite{etheridge}.

In Section \ref{section2} we introduce the space of test functions we will find the bounds for. In Section \ref{section3} we present our main reuslts. Theorem \ref{time_change_theorem} shows how the approach in \cite{diffusion} can be extended to the approximation of a scaled, time-changed random walk by a time-changed Brownian Motion. In Theorem \ref{theorem10.1} we apply Theorem \ref{time_change_theorem} to look at the distance between a time-changed Poisson Process and a time-changed Browian Motion. Theorem \ref{theorem_moran} shows how this can be extended to find the speed of convergence of a process whose dynamics resemble those of the Moran model with mutation to one whose dynamics resemble those of the Wright-Fisher diffusion with mutation. The bound obtained therein makes it possible to analyse the impact the mutation rates and the number of individuals have on the quality of the approximation and the interplay between those parameters. Section \ref{section4} provides proofs of the reults presented in Section \ref{section2} and comments on how the proof of Theorem \ref{theorem_moran} may be adapted to find the speed of convergence in other examples satisfying the Stroock-Varadhan theory of diffusion approximation (see \cite{stroock}).

In what follows, $\|\cdot\|$ will always denote the sup norm and $D=D[0,1]=D\left([0,1],\mathbbm{R}\right)$ will be the Skorokhod space of c\`adl\`ag real-valued functions on $[0,1]$.
\section{Space M}\label{section2}
Let us define:
$$\|f\|_L:=\sup_{w\in D[0,1]}\frac{|f(w)|}{1+\|w\|^3}\text{,}$$
and let $L$ be the Banach space of the continuous functions $f:D[0,1]\to\mathbbm{R}$ such that $\|f\|_L<\infty$. We now let $M\subset L$ consist of the twice Fr\'echet differentiable functions $f$, such that:
\begin{equation}\label{space_m}
\|D^2f(w+h)-D^2f(w)\|\leq k_f\|h\|\text{,}
\end{equation}
for some constant $k_f$, uniformly in $w,h\in D[0,1]$. By $D^kf$ we mean the $k$-th Fr\'echet derivative of $f$ and the $k$-linear norm $B$ on $L$ is defined to be $\|B\|=\sup_{\lbrace h:\|h\|=1\rbrace} |B[h,...,h]|$. Note the following:
\begin{lemma}
For every $f\in M$, $\|f\|_M<\infty$, where:
\begin{align*}
\|f\|_M:=&\sup_{w\in D[0,1]}\frac{|f(w)|}{1+\|w\|^3}+\sup_{w\in D[0,1]}\frac{\|Df(w)\|}{1+\|w\|^2}+\sup_{w\in D[0,1]}\frac{\|D^2f(w)\|}{1+\|w\|}\\
&+\sup_{w,h\in D[0,1]}\frac{\|D^2f(w+h)-D^2f(w)\|}{\|h\|}.
\end{align*}
\end{lemma}
\begin{proof}
Note that for $f\in M$ it is possible to find a constant $K_f$ satisfying:
\begin{align}
\text{A)}&\quad\|Df(w)\|\leq \|Df(w)-Df(0)\|+\|Df(0)\|\nonumber\\
\stackrel{\text{MVT}}{\leq}& \|w\|\sup_{\theta\in[0,1]}\|D^2f(\theta w)\|+\|Df(0)\|\nonumber\\
\leq& \|w\|\left[\sup_{\theta\in[0,1]}\left(\|D^2f(\theta w)-D^2f(0)\|+\|D^2f(0)\|\right)\right]+\|Df(0)\|\nonumber\\
\stackrel{(\ref{space_m})}\leq&\|w\|\left[k_f\sup_{\theta\in[0,1]}\|\theta w\|+\|D^2f(0)\|\right]+\|Df(0)\|\nonumber\\
\leq &k_f\|w\|^2+\|D^2f(0)\|(1\vee\|w\|^2)+\|Df(0)\|< K_f(1+\|w\|^2)\text{;}\nonumber\\
\text{B)}&\quad\|D^2f(w)\|\leq\|D^2f(w)-D^2f(0)\|+\|D^2f(0)\|\nonumber\\
\stackrel{(\ref{space_m})}{\leq}&k_f\|w\|+\|D^2f(0)\|<K_f(1+\|w\|)\text{;}\nonumber\\
\text{C)}&\quad\left|f(w+h)-f(w)-Df(w)[h]-\frac{1}{2}D^2f(w)[h,h]\right|\leq K_f\|h\|^3\text{,}
\label{first_der}
\end{align}
uniformly in $w,h\in D$, where the last inequality follows by Taylor's theorem and $(\ref{space_m})$. Therefore:
\begin{align*}
\|f\|_M=&\sup_{w\in D}\frac{|f(w)|}{1+\|w\|^3}+\sup_{w\in D}\frac{\|Df(w)\|}{1+\|w\|^2}+\sup_{w\in D}\frac{\|D^2f(w)\|}{1+\|w\|}\\
&+\sup_{w,h\in D}\frac{\|D^2f(w+h)-D^2f(w)\|}{\|h\|}<\infty.
\end{align*}
\end{proof}
We now let $M^0\subset M$ be the class of functionals $g\in M$ such that:
$$\|g\|_{M^0}:=\|g\|_M+\sup_{w\in D}|g(w)|+\sup_{w\in D}\|Dg(w)\|+\sup_{w\in D}\|D^2g(w)\|<\infty\text{.}$$
This is Proposition 3.1 of \cite{functional_combinatorial}:
\begin{proposition}\label{prop_m}
Suppose that, for each $n\geq 1$, the random element $Y_n$ of $D[0,1]$ is piecewise constant with intervals of constancy of length at least $r_n$. Let $\left(Z_n\right)_{n\geq 1}$ be random elements of $D[0,1]$ converging weakly in $D[0,1]$, with respect to the Skorokhod topology, to a random element $Z\in C\left([0,1],\mathbbm{R}\right)$. If:
\begin{equation}\label{assumption}
|\mathbbm{E}g(Y_n)-\mathbbm{E}g(Z_n)|\leq C\tau_n\|g\|_{M^0}
\end{equation}
for each $g\in M^0$ and if $\tau_n\log^2(1/r_n)\xrightarrow{n\to\infty}0$, then $Y_n\Rightarrow Z$ in $D[0,1]$ (weakly in the Skorokhod topology).
\end{proposition}
A similar result holds when $Y_n$ is a continuous-time Markov chain:
\begin{proposition}\label{prop_m1}
Suppose that, for each $n\geq 1$, the random element $Y_n$ of $D[0,1]$ is a contiuous-time Markov chain with mean holding time $\frac{1}{\lambda_n}\to 0$, identically distributed for each state. Let $\left(Z_n\right)_{n\geq 1}$ be random elements of $D[0,1]$ converging weakly in $D[0,1]$, with respect to the Skorokhod topology, to a random element $Z\in C\left([0,1],\mathbbm{R}\right)$. Suppose further that:
\begin{equation}\label{assumption1}
|\mathbbm{E}g(Y_n)-\mathbbm{E}g(Z_n)|\leq C\tau_n\|g\|_{M^0}
\end{equation}
for each $g\in M^0$ and that $\tau_n\log^2\left((\lambda_n)^3\right)\xrightarrow{n\to\infty}0$. Then $Y_n\Rightarrow Z$ in $D[0,1]$ (weakly in the Skorokhod topology).
\end{proposition}
We provide a proof of Proposition \ref{prop_m1} in the Appendix.
\section{Main results}\label{section3}
Theorem \ref{time_change_theorem} below is an extension of Theorem 1 in \cite{diffusion} to the case of a time-changed scaled random walk: 
\begin{theorem}\label{time_change_theorem}
Let $X_1,X_2,...$ be i.i.d. with mean $0$, variance $1$ and finite third moment. Let $s:[0,1]\to[0,\infty)$ be a strictly increasing, continuous function with $s(0)=0$. Define:
$$Y_n(t)=n^{-1/2}\sum_{i=1}^{\lfloor ns(t)\rfloor} X_i,\quad t\in[0,1]$$
and let $(Z(t),t\in[0,1])=\left(B(s(t)),t\in[0,1]\right)$, where $B$ is a standard Brownian Motion. Suppose that $g\in M$.
Then:
\begin{align*}
|\mathbbm{E}g(Y_n)-\mathbbm{E}g(Z)|\leq &\|g\|_M\frac{30+54\cdot 5^{1/3}s(1)}{\sqrt{\pi\log 2}}n^{-1/2}\sqrt{\log (2s(1)n)}\\
&+\|g\|_Ms(1)\left(1+\left(\frac{3}{2}\right)^3\sqrt{\frac{2}{\pi}}s(1)^{3/2}\right)\mathbbm{E}|X_1|^3n^{-1/2}\\
&+\|g\|_M\frac{2160}{\sqrt{\pi}(\log 2)^{3/2}}n^{-3/2}(\log (2s(1)n))^{3/2}.
\end{align*}
\end{theorem}
In Theorem \ref{time_change_theorem} we do not claim that our bounds are sharp. Our bound in Theorem \ref{time_change_theorem} is of the same order as the one obtained in the original case in \cite{diffusion}. This result can also be extended in a straightforward way to instances in which the time change is random and independent of the step sizes of the random walk. We can obtain this by conditioning on the time change.

Theorem \ref{theorem10.1} below treats a time-changed Poisson process and can also be extended to random time changes, independent of the Poisson process of interest, by conditioning.
\begin{theorem}\label{theorem10.1}
Suppose that $P$ is a Poisson process with rate $1$ and $S^{(n)}:[0,1]\to[0,\infty)$ is a sequence of increasing continuous functions, such that $S^{(n)}(0)=0$. Let $S:[0,1]\to[0,\infty)$ be also increasing and continuous. Let $Z(t)=B(S(t)),t\in[0,1]$ where $B$ is a standard Brownian Motion and 
$$\tilde{Y}_n(t)=\frac{P\left(nS^{(n)}(t)\right)-nS^{(n)}(t)}{\sqrt{n}},\qquad t\in[0,1].$$
Then, for all $g\in M$:
\begin{equation*}
\begin{aligned}
\MoveEqLeft|\mathbbm{E}g(\tilde{Y}_n)-\mathbbm{E}g(Z)|\leq \|g\|_M\left\lbrace \left(2+\frac{27\sqrt{2}}{2\sqrt{\pi}}S(1)\right)\sqrt{\|S-S^{(n)}\|}+\frac{27\sqrt{2}}{2\sqrt{\pi}}\|S-S^{(n)}\|^{3/2}\right.\\
&+ n^{-1/2}\left[\frac{30+54\cdot 5^{1/3}S(1)}{\sqrt{\pi\log 2}}\sqrt{\log (2s(1)n)}\right.\\
&+\left(1+\left(\frac{3}{2}\right)^3\sqrt{\frac{2}{\pi}}S^{(n)}(1)^{3/2}\right)S^{(n)}(1)(1+2e^{-1})\left.+1+\frac{\log(1+2e^{-1})+2\log n}{\log\log(n+2)}\right]\\
&+n^{-1}\frac{9\sqrt{S^{(n)}(1)}}{2}\left(1+3nS^{(n)}(1)\right)^{1/2}\left[4+\frac{16701+128(\log n)^3}{(\log\log(n+3))^3}\right]^{1/3}\\
&+\left.n^{-3/2}\left[\frac{2160}{\sqrt{\pi}(\log 2)^{3/2}}(\log (2S(1)n))^{3/2}+8+\frac{33402+256(\log n)^3}{(\log\log(n+3))^3}\right]\right\rbrace.
\end{aligned}
\end{equation*}
\end{theorem}
\begin{remark}
The bound in Theorem \ref{theorem10.1} goes to $0$ as long as the time changes $S^n\to S$ uniformly.
\end{remark}
Theorem \ref{theorem_moran} below gives a bound on the speed of convergence of a process whose dynamics are similar to those of the Moran model with mutation to a process whose dynamics are similar to those of the Wright-Fisher diffusion with mutation. In the Moran model with mutation, in a population of size $n$, each individual carries a particular gene of one of the two forms: $A$ and $a$. Each individual has exactly one parent and offspring inherit the genetic type of their parent. Now, at exponential rate ${n\choose 2}$ a pair of genes is sampled uniformly at random from the population. One of the pair is selected at random to die and the other one splits in two. In addition, every individual of type A changes its type independently at rate $\nu_2$ and every individual of type a changes its type independently at rate $\nu_1$. 
\begin{theorem}\label{theorem_moran}
Let $X_n(t)$ be the proportion of type $a$ genes in the population at time $t\in[0,1]$ under the Moran model with mutation rates $\nu_1,\nu_2$, as described above. Let $(X(t),t\in[0,1])$ denote the Wright-Fisher diffusion given by:
$$dX(t)=(\nu_2-(\nu_1+\nu_2)X(t))dt+\sqrt{X(t)(1-X(t))}dB_t.$$

Suppose that:
\begin{equation}\label{intens}
\begin{cases}
\mathbbm{P}\left[X_n(t+\Delta t)-X_n(t)=\frac{1}{n}\right]=n^2R_1^{(n)}(t)\Delta t\\
\mathbbm{P}\left[X_n(t+\Delta t)-X_n(t)=\frac{1}{n}\right]=n^2R_{-1}^{(n)}(t)\Delta t
\end{cases}.
\end{equation}
and:
\begin{equation}\label{m_n_th}
M_n(t)=\frac{1}{n}P_1\left(n^2R_1^{(n)}(t)\right)-\frac{1}{n}P_{-1}\left(n^2R_{-1}^{(n)}(t)\right),
\end{equation}
for $t\in[0,1]$, where $P_1$ and $P_{-1}$ are i.i.d. rate $1$ Poisson processes, independent of $X_n$.

Suppose further that 
\begin{equation}\label{m_th}
M(t)=W\left(\int_0^tX(t)(1-X(t))ds\right)+\int_0^t(\nu_2-(\nu_1+\nu_2)X(s))ds,
\end{equation}
where $W$ is a standard Brownian Motion, independent of $X$.
Then, for any $g\in M$:
\begin{align*}
&\left|\mathbbm{E}g(M_n)-\mathbbm{E}g(M)\right|\\
\leq& \|g\|_M\left\lbrace \left(18+\nu_1^{1/2}+47\nu_1^{3/4}+31\nu_1^{3/2}+\nu_2+3\nu_2^2+9\nu_2^3\right)\right.\\
&\cdot\left(1.02\cdot 10^6+425\nu_2^{1/2}+623\nu_2+39\nu_2^{3/2}+7\nu_2^{5/2}\right)\\
&+ \left(12+3\nu_2+3\nu_2^2+9\nu_2^3\right)\left(1.02\cdot 10^6+425\nu_1^{1/2}+623\nu_1+39\nu_1^{3/2}+7\nu_1^{5/2}\right)\\
&\left.+7\left(\frac{1}{2}(1+2\nu_2)(\nu_1+\nu_2)+31(\nu_1+\nu_2)^3\right)\right\rbrace n^{-1/4}\\
&+\|g\|_M2112\left[\left(18+\nu_1^{1/2}+47\nu_1^{3/4}+31\nu_1^{3/2}+\nu_2+3\nu_2^2+9\nu_2^3\right)\left(\log\left(n^2/4+\nu_2n\right)\right)^{3/2}\right.\\
&+\left.\left(12+3\nu_2+3\nu_2^2+9\nu_2^3\right)\left(\log\left(n^2/4+\nu_1n\right)\right)^{3/2}\right]n^{-3}.
\end{align*}
\end{theorem}
\begin{remark}
If $\nu_1\geq 1$ and $\nu_2\geq 1$ then we can write:
\begin{align*}
&|\mathbbm{E}g(M_n)-\mathbbm{E}g(M)|\\
\leq& \|g\|_M\left[(18+79\nu_1^{3/2}+13\nu_2^3)(1.02\cdot 10^6+1094\nu_2^{5/2})\right.\\
&+(12+15\nu_2^3)(1.02\cdot 10^6+1094\nu_2^{5/2})\\
&+\left.7\left(31.5\nu_1^3+32.5\nu_2^3+\nu_1\nu_2(1+93\nu_1+93\nu_2)\right)\right]n^{-1/4}\\
&+\|g\|_M2112(18+79\nu_1^{3/2}+13\nu_2^{3})n^{-3}\left(\log\left(n^2/4+\nu_1n\right)\right)^{3/2}\\
&+\|g\|_M2112(12+15\nu_2^3)n^{-3}\left(\log\left(n^2/4+\nu_1n\right)\right)^{3/2}.
\end{align*}
The approximation gets worse as the mutation rates increase and the number of individuals decreases. Should we want to make the mutation rates depend on $n$ and be of the same order, we will require them to be $o\left(n^{1/22}\right)$ in order for the bound to converge to $0$ as $n\to\infty$.
\end{remark}
\begin{remark}
Note that the Moran model $X_n$ jumps up by $\frac{1}{n}$ with intensity $n^2R_1^{(n)}(t)$ and down by $\frac{1}{n}$ with intensity $n^2R_{-1}^{(n)}(t)$, as defined in (\ref{intens}). Using the ideas from \cite{kurtz} we can write $X_n$ in the following form:
$$X_n(t)=\frac{1}{n}\tilde{P}_1\left(n^2R_1^{(n)}(t)\right)-\frac{1}{n}\tilde{P}_{-1}\left(n^2R_{-1}^{(n)}(t)\right)$$
for some indepenent rate $1$ Poisson processes $\tilde{P}_1$ and $\tilde{P}_{-1}$ which are, however, not independent of $R_1^{(n)}$ and $R_{-1}^{(n)}$. In Theorem \ref{theorem_moran} we consider a similar process, given by (\ref{m_n_th}), whose definition uses Poisson processes $P_1$ and $P_{-1}$ independent of $R_1^{(n)}$ and $R_{-1}^{(n)}$.

Similarly, the diffusive part of the Wright-Fisher diffusion may be expressed as a time-changed Brownian Motion $\tilde{W}\left(\int_0^tX(t)(1-X(t)ds\right)$. However, we consider a process given by (\ref{m_th}) with the assumption that $W$ is independent of $X$.

We will  appeal to Theorem \ref{theorem10.1} to obtain the bounds. The time changes we apply in this case are random and therefore we will first condition on them.
\end{remark}
\begin{remark}
The Moran model $M_n$ converges weakly to the Wright-Fisher diffusion $M$, which can be proved using, for instance, the Stroock-Varadhan theory of diffusion approximation \cite{stroock}. However, our paper does not provide the tools necessary for treating this convergence with Stein's method and obtaining bounds on its rate.
\end{remark}
A key idea used in the proof will be the Donnelly-Kurtz look-down construction coming from \cite{donnelly_kurtz} and decribed in Chapter 2.10 of \cite{etheridge}. The $n$-particle look-down process is denoted by a vector $\left(\psi_1(t),...,\psi_n(t)\right)$ with each index representing a "level" and each of $\psi_i$'s representing the type of the individual at level $i$ at time $t$. Individual at level $k$ is equipped with an exponential clock with rate $k-1$, independent of other individuals, and at the times the clock rings it "looks down" at a level chosen uniformly at random from $\lbrace 1,...,k-1\rbrace$ and adopts the type of the individual at that level. In addition, the type of each individual evolves according to the mutation process. A comparison of the generators of the Moran model and the look-down process shows that, as long as the two are started from the same initial exchangeable condition, they produce the same distribution of types in the population. In addition, it may be shown that the Wright-Fisher diffusion may be represented as the proportion of type a individuals in the population, in which the types are distributed according to the infinite-particle look-down process. The corresponding Moran model is then the proportion of type a individuals among the ones located on the first $n$ levels. Due to exchangeability, we may then describe the Moran model $X_n$ at a fixed time, as depending on the Wright-Fisher diffusion $X$ in the following way: $nX_n(t)\sim\textbf{Binomial}(n,X(s))$.
\begin{remark}
Our bound in Theorem \ref{theorem_moran} is sufficient to conclude that process $M_n$ converges to process $M$ weakly on compact intervals. This follows from Proposition \ref{prop_m1}. Using the notation therein, in this case, $\tau_n=n^{-1/4}$ and $\lambda_n=\frac{n(n-1)}{2}$.
\end{remark}
\section{Setting up Stein's method}
Let us first define:
\begin{equation}\label{a_n}
A_n(t)=n^{-1/2}\sum_{i=1}^{\lfloor ns(1)\rfloor} Z_i\mathbbm{1}_{[i/n,s(1)]}(s(t))=n^{-1/2}\sum_{i=1}^{\lfloor ns(1)\rfloor} Z_i\mathbbm{1}_{[s^{-1}(i/n),1]}(t),
\end{equation}
for $Z_i\stackrel{\text{i.i.d}}\sim \mathcal{N}(0,1)$.
In a preparation for the proof of Theorem \ref{time_change_theorem}, we will apply Stein's method to find the distance between $A_n$ and $Y_n$.
\subsection{The Stein equation}
We first note that if $U_1,U_2,...$ are i.i.d. Ornstein-Uhlenbeck processes with stationary law $\mathcal{N}(0,1)$, then defining:
$$W_n(t,u)=n^{-1/2}\sum_{i=1}^{\lfloor ns(t)\rfloor} U_i(u),\quad u\geq 0,t\in[0,1],$$
we obtain that the law of $A_n$ is stationary for $\left(W_n(\cdot,u)\right)_{u\geq 0}$. Denote the generator of $\left(W_n(\cdot,u)\right)_{u\geq 0}$ by $\mathcal{A}_n$. By properties of stationary distributions,  $\mathbbm{E}_{\mu}\mathcal{A}_nf=0$ for all $f\in\text{Domain}(\mathcal{A}_n)$ if and only if $\mu=\mathcal{L}(A_n)$. Therefore, we can treat
\begin{equation} \label{stein_eq}
\mathcal{A}_nf=g-\mathbbm{E}g(A_n)
\end{equation}
as our Stein equation.

In the next subsection, for any $g$ from a suitable class of functions, we will find an $f$ satisfying equation (\ref{stein_eq}). Then, in the sequel, we will find a bound on $|\mathbbm{E}\mathcal{A}_nf(Y_n)|$, which will readily give us a bound on $|\mathbbm{E}g(Y_n)-\mathbbm{E}g(A_n)|$.

\begin{proposition}\label{approx_proposition}
The generator $\mathcal{A}_n$ of the process $(W_n(\cdot,u))_{u\geq 0}$ acts on any $f\in M$ in the following way:
$$(\mathcal{A}_nf)(w):=-Df(w)[w]+\mathbbm{E}D^2f(w)\left[A_n^{(2)}\right].$$
\end{proposition}
Before we prove this result, we need a lemma:
\begin{lemma}\label{second_ad_lemma}
Letting $\mathcal{F}_{n,u}=\sigma(W_n(\cdot,v),v\leq u)$, we have:
$$W_n(\cdot,u+v)-e^{-v}W_n(\cdot,u)\stackrel{\mathcal{D}}{=}\sigma(v)A_n(\cdot).$$
\end{lemma}
\begin{proof} We first note that for each $i\geq 1$ we can construct independent standard Brownian Motions $B_i$ such that $(X_i(u),u\geq 0)=(e^{-u}B_i(e^{2u}),u\geq 0)$. Then:

\begin{align*}
W_n(\cdot,u+v)-e^{-v}W(\cdot,u)=&n^{-1/2}\sum_{k=1}^{\lfloor ns(\cdot)\rfloor} U_k(u+v)-e^{-v}n^{-1/2}\sum_{k=1}^{\lfloor ns(\cdot)\rfloor} U_k(u)\\
\stackrel{\mathcal{D}}{=}&n^{-1/2}e^{-(u+v)}\sum_{k=1}^{\lfloor ns(\cdot)\rfloor}\left[B_k\left(e^{2(u+v)}\right)-B_k\left(e^{2u}\right)\right]\\
\stackrel{\mathcal{D}}{=}&n^{-1/2}\sigma(v)\sum_{k=1}^{\lfloor ns(\cdot)\rfloor}Z_k=\sigma(v)A_n(\cdot)\text{.}
\end{align*} 
\end{proof}
\begin{proof}[Proof of Proposition \ref{approx_proposition}]
Note that the semigroup of  $(W_n(\cdot,u))_{u\geq 0}$, acting on $L$, is defined by:
\begin{equation*}
(T_{n,u}f)(w):=\mathbbm{E}\left[f(W_n(\cdot,u)|W_n(\cdot,0)=w\right]
\end{equation*}
and by Lemma \ref{second_ad_lemma} we readily obtain that:
\begin{equation}\label{semigroup}
(T_{n,u}f)(w)=\mathbbm{E}\left[f(we^{-u}+\sigma(u)A_n(\cdot)\right]\text{.}
\end{equation}
We can define the generator by: $\mathcal{A}_n=\lim_{u\searrow 0}\frac{T_{n,u}-I}{u}$. We also have that for $f\in M$:
\begin{align*}
&\left|(T_{n,u}f)(w)-f(w)-\mathbbm{E}Df(w)[\sigma(u)A_n-w(1-e^{-u})]\right.\\
&-\left.\frac{1}{2}\mathbbm{E}D^2f(w)\left[\lbrace\sigma(u)A_n-w(1-e^{-u})\rbrace^{(2)}\right]\right|\\
\stackrel{(\ref{first_der})}{\leq}&K_f\mathbbm{E}\|\sigma(u)A_n-w(1-e^{-u})\|^3\\
\leq &4K_f\left[\sigma^3(u)\mathbbm{E}\|A_n\|^3+(1-e^{-u})^3\|w\|^3\right]
\leq K_3(1+\|w\|^3)u^{3/2}\\
\end{align*}
for a constant $K_3$ depending only on $f$, where the last inequality follows from the fact that for $u\geq 0$, $\sigma^3(u)\leq 3u^{3/2}$ and $(1-e^{-u})^3\leq u^{3/2}$.
So:

\begin{align}
&\left|(T_{n,u}f-f)(w)+uDf(w)[w]-u\mathbbm{E}D^2f(w)[A_n^{(2)}]\right|\nonumber \\
\leq&\left|(T_{n,u}f)(w)-f(w)-\mathbbm{E}Df(w)[\sigma(u)A_n-w(1-e^{-u})]\right.\nonumber\\
&-\left.\frac{1}{2}\mathbbm{E}D^2f(w)[\lbrace\sigma(u)A_n-w(1-e^{-u})\rbrace^{(2)}]\right|+\left|\sigma(u)\mathbbm{E}Df(w)[A_n]\right|\nonumber\\
&+\left|(u-1+e^{-u})Df(w)[w]\right|+\left|\left(\frac{\sigma^2(u)}{2}-u\right)\mathbbm{E}D^2f(w)[A_n^{(2)}]\right|\nonumber\\
&+\left|\frac{(1-e^{-u})^2}{2}D^2f(w)[w^{(2)}]\right|+\left|\sigma(u)(1-e^{-u})\mathbbm{E}D^2f(w)[A_n,w]\right|\nonumber\\
\leq &K_1\left(1+\|w\|^3\right)u^{3/2}+\left|\sigma(u)\mathbbm{E}Df(w)[A_n]\right|+|u-1+e^{-u}|\|Df(w)\|\|w\|\nonumber\\
&+\left|\frac{\sigma^2(u)}{2}-u\right|\|D^2f(w)\|\mathbbm{E}\|A_n\|^2\nonumber\\
&+\frac{(1-e^{-u})^2}{2}\|D^2f(w)\|\|w\|^2+\sigma(u)(1-e^{-u})\|D^2f(w)\|\|w\|\mathbbm{E}\|A_n\|\nonumber\\
\stackrel{(\ref{first_der})}\leq &K_1(1+\|w\|^3)u^{3/2}+\left|\sigma(u)\mathbbm{E}Df(w)[A_n]\right|+|u-1+e^{-u}|K_f(1+\|w\|^2)\|w\|\nonumber\\
&+\left|\frac{\sigma^2(u)}{2}-u\right|K_f(1+\|w\|)\mathbbm{E}\|A_n\|^2+\frac{(1-e^{-u})^2}{2}K_f(1+\|w\|)\|w\|^2\nonumber\\
&+\sigma(u)(1-e^{-u})(1+\|w\|)\|w\|\mathbbm{E}\|A_n\|\nonumber\\
\leq &3u^{3/2}\left(K_1(1+\|w\|^3)+K_f(1+\|w\|^2)\|w\|+K_f(1+\|w\|)\mathbbm{E}\|A_n\|^2\right.\nonumber\\
&\left.+K_f(1+\|w\|)\|w\|^2+(1+\|w\|)\|w\|\mathbbm{E}\|A_n\|\right)+\left|\sigma(u)\mathbbm{E}Df(w)[A_n]\right|\nonumber\\
\leq& K_4(1+\|w\|^3)u^{3/2}\text{,}
\label{trala}
\end{align}
for some constant $K_4$ depending only on $f$. The last inequality follows from the fact that:
\begin{align*}
\mathbbm{E}Df(w)[A_n]=&\mathbbm{E}Df(w)\left[n^{-1/2}\sum_{k=1}^{\lfloor ns(1)\rfloor}Z_k
\mathbbm{1}_{[s^{-1}(k/n),1]}\right]\\
=&n^{-1/2}\sum_{k=1}^{\lfloor ns(1)\rfloor}Df(w)\left[\mathbbm{1}_{[s^{-1}(k/n),1]}\right]\mathbbm{E}[Z_k]\\
=&0\text{.}
\end{align*}
It follows that for any $f\in M$:
$$\mathcal{A}_nf(w)=\lim_{u\searrow0}\frac{T_{n,u}f(w)-f(w)}{u}=-Df(w)[w]+\mathbbm{E}D^2f(w)\left[A_n^{(2)}\right]\text{.}$$
\end{proof}
\subsection{Solving Stein's equation}
\begin{proposition}\label{prop2.10}
Suppose that $g\in M$ satisfies $\mathbbm{E}g(A_n)=0$. Then the equation: $\mathcal{A}_nf=g$
is solved by:
\begin{equation}\label{phi}
f=\phi_n(g)=-\int_0^\infty T_{n,u}gdu
\end{equation}
for $T_{n,u}$ defined by (\ref{semigroup}).
Furthermore, $\phi_n(g)\in M$ and the following inequalities hold:
\begin{align}
&\text{A)}\quad\|D\phi_n(g)(w)\| \leq\left(1+\frac{4}{3}\mathbbm{E}\|A_n\|^2\right)\|g\|_M(1+\|w\|^2)\text{,}\nonumber\\
&\text{B)}\quad\|D^2\phi_n(g)(w)\|\leq \left(\frac{1}{2}+\frac{\mathbbm{E}\|A_n\|}{3}\right)\|g\|_M(1+\|w\|)\text{,}\nonumber\\
&\text{C)}\quad\|D^2\phi_n(g)(w+h)-D^2\phi(g)(w)\|\leq\frac{\|g\|_M}{3}\|h\|\text{.}\label{norm_bound}
\end{align}
\end{proposition}
\begin{proof}
The proof will follow the procedure used to prove \cite[Proposition 1]{kasprzak}.

\textbf{Step 1}: First, we show that $\phi_n(g)\in M$ and that (\ref{norm_bound}) holds. Assume $g\in M$ and $\mathbbm{E}[g(A_n)]=0$. 
Note that if, for instance:
\begin{equation}\label{condition}
|g(w)-g(x)|\leq C_g(1+\|w\|^2+\|x\|^2)\|w-x\|
\end{equation}
uniformly in $w,x\in D[0,1]$ then:
\begin{align}
&\lim_{t\to\infty}\int_0^t\left|T_{n,u}g(w)\right|du=\lim_{t\to\infty}\int_0^t\left|\mathbbm{E}g(we^{-u}+\sigma(u)A_n)\right|du\nonumber\\
\leq&\lim_{t\to\infty}\left[\int_0^t\left|\mathbbm{E}\left[g(we^{-u}+\sigma(u)A_n)-g(\sigma(u)A_n)\right]\right|du+\int_0^t\left|\mathbbm{E}[g(\sigma(u)A_n)-g(A_n)]\right|du\right]\nonumber\\
\leq &C_g\lim_{t\to\infty}\left[\int_0^t\mathbbm{E}\left[\left(1+\|e^{-u}w+\sigma(u)A_n\|^2+\sigma^2(u)\|A_n\|^2\right)e^{-u}\|w\|\right]du\right.\nonumber\\
&\left.+\int_0^t\mathbbm{E}\left|(1+(\sigma^2(u)+1)\|A_n\|^2)\right|\left\|(\sigma(u)-1)A_n\right\|du\right] \nonumber\\
\leq& C_g\lim_{t\to\infty}\left[\int_0^t\left[ e^{-u}\|w\|+2e^{-3u}\|w\|^3+3\sigma^2(u)e^{-u}\|w\|\mathbbm{E}\|A_n\|^2\right]du\right.\nonumber\\
&\left.+\int_0^t(\sigma(u)-1)\mathbbm{E}\left|(1+(\sigma^2(u)+1)\|Z\|^2)\right|\left\|A_n\right\|du\right]\nonumber\\
\leq&C(1+\|w\|^3)\text{,}
\label{4.77}
\end{align}
for some constant $C$. For such $g$,
$f=\phi_n(g)=-\int_0^\infty T_{n,u}gdu$
exists.
Note that every $g\in M$ satisfies condition (\ref{condition}) because for such $g$:
\begin{align*}
&|g(w)-g(x)|\stackrel{(\ref{first_der})}\leq K_g\|w-x\|^3+\left|Dg(x)[w-x]+\frac{1}{2}D^2g(x)[w-x,w-x]\right| \\
\leq& K_g\|w-x\|^3+\|Dg(x)\|\|w-x\|+\frac{1}{2}\|D^2g(x)\|\|w-x\|^2 \\
\stackrel{(\ref{first_der})}\leq& K_g\|w-x\|\left(\|w-x\|^2+1+\|x\|^2+\frac{1}{2}\|w-x\|(1+\|x\|)\right)\\
\leq& K_g\|w-x\|\left(2\|w\|^2+2\|x\|^2+1+\|x\|^2+\frac{1}{2}(\|w\|+\|x\|+\|w\|\|x\|+\|x\|^2)\right)\\
\leq& C_g(1+\|w\|^2+\|x\|^2)\|w-x\|
\end{align*}
uniformly in $w,x$ because $\|w\|\leq 1+\|w\|^2$, $\|x\|\leq 1+\|x\|^2$ and $\|w\|\|x\|\leq \|w\|^2+\|x\|^2$. 

Now take $g\in M$, such that $\mathbbm{E}g(A_n)=0$ and note that for $\phi_n$ defined in (\ref{phi}) we get:
$$\phi_n(g)(w+h)-\phi_n(g)(w)\stackrel{(\ref{semigroup})}=-\mathbbm{E}\int_0^\infty\left[g\left((w+h)e^{-u}+\sigma(u)A_n\right)-g\left(we^{-u}+\sigma(u)A_n\right)\right]du$$
and so dominated convergence (which can be applied because of (\ref{4.77})) gives:
\begin{equation}\label{4.800}
D^k\phi_n(g)(w)=-\mathbbm{E}\int_0^{\infty}e^{-ku}D^kg(we^{-u}+\sigma(u)A_n)du,\quad k=1,2\text{.}
\end{equation}

Also, observe that:
\begin{align}
&\text{A)}\quad\|D\phi_n(g)(w)\|\nonumber\\
&\leq \int_{0}^\infty e^{-u}\mathbbm{E}\|Dg(we^{-u}+\sigma(u)A_n)\|du\nonumber\\
&\stackrel{(\ref{first_der})}\leq \|g\|_M\int_0^\infty e^{-u}(1+\mathbbm{E}\|we^{-u}+\sigma(u)A_n\|^2)du \nonumber\\
&\leq \|g\|_M\int_0^\infty \left(e^{-u}+2\|w\|^2e^{-3u}+2\mathbbm{E}\|A_n\|^2(e^{-u}-e^{-3u})\right)du\nonumber\\
&\leq \left(1+\frac{4}{3}\mathbbm{E}\|A_n\|^2\right)\|g\|_M(1+\|w\|^2)\text{,}\nonumber\\
&\text{B)}\quad\|D^2\phi_n(g)(w)\|\nonumber\\
&\leq \int_0^\infty e^{-2u}\mathbbm{E}\left\|D^2g(we^{-u}+\sigma(u)A_n)\right\|du\nonumber\\
&\stackrel{(\ref{first_der})}\leq \|g\|_M\int_0^\infty e^{-2u}(1+\mathbbm{E}\|we^{-u}+\sigma(u)A_n\|)du\nonumber\\
&\leq\|g\|_M\int_0^\infty e^{-2u}(1+e^{-u}\|w\|+\sigma(u)\mathbbm{E}\|A_n\|)du\nonumber\\
&\leq \|g\|_M\left(\frac{1}{2}+\frac{\|w\|}{3}+\frac{\mathbbm{E}\|A_n\|}{3}\right)\nonumber\\
&\leq \left(\frac{1}{2}+\frac{\mathbbm{E}\|A_n\|}{3}\right)\|g\|_M(1+\|w\|)\text{,}\nonumber\\
&\text{C)}\quad\|D^2\phi_n(g)(w+h)-D^2\phi_n(g)(w)\|\nonumber\\
&\leq\left\|\mathbbm{E}\int_0^\infty e^{-2u}D^2g((w+h)e^{-u}+\sigma(u)A_n)-e^{-2u}D^2g(we^{-u}+\sigma(u)A_n)du\right\|\nonumber\\
&\stackrel{g\in M,(\ref{space_m})}\leq\|g\|_M\int_0^\infty e^{-2u}e^{-u}\|h\|du=\frac{\|g\|_M}{3}\|h\|\text{,}
\label{4.78}
\end{align}
uniformly in $g\in M$, which proves (\ref{norm_bound}). It follows by (\ref{4.77}) and (\ref{4.78}) that $\phi_n(g)\in M$. Thus, $\phi_n(g)$ is in the domain of the semigroup and of $\mathcal{A}_n$. Also, similarly, for any $t>0$, $\int_0^tT_{n,u}gdu\in M$.

\textbf{Step 2}: We now show that for all $t>0$ and for all $g\in M$:
\begin{equation}\label{4.100}
T_{n,t}g-g=\mathcal{A}_n\left(\int_0^t T_{n,u}gdu\right)\text{.}
\end{equation}
We will prove it by following the steps of the proof of Proposition 1.5 on p. 9 of \cite{ethier}. Observe that for all $w\in D[0,1]$ and $h>0$:
\begin{align}
&\frac{1}{h}[T_{n,h}-I]\int_0^tT_{n,u}g(w)du=\frac{1}{h}\int_0^t[T_{n,u+h}g(w)-T_{n,u}g(w)]du\nonumber\\
=&\frac{1}{h}\int_t^{t+h}T_{n,u}g(w)du-\frac{1}{h}\int_0^hT_{n,u}g(w)du\nonumber\\
\stackrel{(\ref{semigroup})}=&\frac{1}{h}\int_t^{t+h}\mathbbm{E}[g(w e^{-u}+\sigma(u)A_n)]du-\frac{1}{h}\int_0^{h}\mathbbm{E}[g(w e^{-u}+\sigma(u)A_n)]du.
&\label{4.600}
\end{align}
Taking $h\to 0$ on the left-hand side gives $\mathcal{A}_n\left(\int_0^t T_{n,u}g(w)du\right)$ because $\int_0^t T_{n,u}g(w)du$ is in the domain of $\mathcal{A}_n$, as shown in \textbf{Step 1}. In order to analyse the right-hand side note that:

\begin{align}
&\left|\frac{1}{h}\int_0^h\mathbbm{E}[g(w e^{-u}+\sigma(u)A_n)]-g(w)du\right|\nonumber\\
\stackrel{\text{MVT}}\leq &\frac{1}{h}\int_0^h\mathbbm{E}\left[\|w(e^{-u}-1)+\sigma(u)A_n\|\sup_{c\in [0,1]}\|Dg\left(cw+(1-c)(we^{-u}+\sigma(u)A_n)\right)\|\right]du\nonumber\\
\leq&\frac{\|g\|_M}{h}\int_0^h\mathbbm{E}\left[\left(\|w\|(1-e^{-u})+\sigma(u)\|A_n\|\right)\left(1+3\|w\|^2+3\|w\|^2e^{-2u}+3\sigma^2(u)\|A_n\|^2\right)\right]du\nonumber\\
=&\frac{\|g\|_M}{h}\mathbbm{E}\left\lbrace\left(1+3\|w\|^2+3\|A_n\|^2\right)\left(\|w\|(-1+h+\cosh(h)-\sinh(h))\right.\right.\nonumber\\
&\left.+\|A_n\|e^{-h}(-\sqrt{e^{2h}-1}+e^h(h+\log(1+e^{-h}\sqrt{-1+e^{2h}}))\right)\nonumber\\
&+3\|w\|(\|w\|^2-\|A_n\|^2)\left(\frac{e^{-3h}}{6}(e^h-1)^2(e^h+2)\right)\nonumber\\
&+\left.3(\|w\|^2\|A_n\|-\|A_n\|^3)\frac{1}{3}\left(\sqrt{1-e^{-2h}}-\sqrt{e^{-6h}(e^{2h}-1)}\right)\right\rbrace\xrightarrow{h\to 0}0.
\label{4.200}
\end{align}
Similarly:
$$\left|\frac{1}{h}\int_t^{t+h}\mathbbm{E}\left[g(we^{-u}+\sigma(u)A_n)\right]du-\mathbbm{E}\left[g(we^{-t}+\sigma(t)A_n)\right]\right|\xrightarrow{h\to 0}0$$
Therefore, as $h\to 0$, the right-hand side of (\ref{4.600}) converges to $T_{n,t}g-g$, which proves (\ref{4.100}).

\textbf{Step 3}: We note that for any $h>0$ and for any $f\in M$:
$$\frac{1}{h}\left[T_{n,s+h}f-T_{n,s}f\right]=T_{n,s}\left[\frac{T_{n,h}-I}{h}f\right]$$
and therefore for any $w\in D[0,1]$, by dominated convergence (which can be applied because of (\ref{trala})):
\begin{align*}
\left(\frac{d}{ds}\right)^+T_{n,s}f(w)&=\lim_{h\searrow 0}T_{n,s}\left[\frac{T_{n,h}-I}{h}f(w)\right]=\lim_{h\searrow 0}\mathbbm{E}\left[\frac{T_{n,h}-I}{h}f(we^{-s}+\sigma(s)A_n)\right]\\
&=\mathbbm{E}\left[\lim_{h\searrow 0}\frac{T_{n,h}-I}{h}f(we^{-s}+\sigma(s)A_n)\right]=T_{n,s}\mathcal{A}_nf(w)
\end{align*}
and, similarly, for $s>0$, $\left(\frac{d}{ds}\right)^-T_{n,s}f=T_{n,s}\mathcal{A}_nf$ because:
\begin{align*}
&\lim_{h\searrow0}\frac{1}{-h}\left[T_{n,s-h}f-T_{n,s}f\right](w)-T_{n,s}\mathcal{A}_nf(w)\\
=&\lim_{h\searrow0}T_{n,s-h}\left[\left(\frac{T_{n,h}-I}{h}-\mathcal{A}_n\right)f\right](w)+\lim_{h\searrow0}\left(T_{n,s-h}-T_{n,s}\right)\mathcal{A}_nf(w)\\
=&\lim_{h\searrow 0}\mathbbm{E}\left[\left(\frac{T_{n,h}-I}{h}-\mathcal{A}_n\right)f(we^{-s+h}+\sigma(s-h)A_n)\right]\\
&+\lim_{h\searrow 0}\mathbbm{E}\left[\mathcal{A}_nf(we^{-s+h}+\sigma(s-h)A_n)-\mathcal{A}_nf(we^{-s}+\sigma(s)A_n)\right]\\
\stackrel{(\ref{trala})}=&0
\end{align*}
again, by dominated convergence. It can be applied because of (\ref{trala}) and the observation that:
\begin{align*}
&\left|\mathcal{A}_nf(we^{-s+h}+\sigma(s-h)A_n)-\mathcal{A}_nf(we^{-s}+\sigma(s)A_n)\right|\\
=&\left|-Df(we^{-s+h}+\sigma(s-h)A_n)[we^{-s+h}+\sigma(s-h)A_n]\right.\\
&+\sum_{i=1}^nD^2f(we^{-s+h}+\sigma(s-h)A_n)[\mathbbm{1}_{[i/n,1]}^{(2)}]\\
&\left.-Df(we^{-s}+\sigma(s)A_n)[we^{-s}+\sigma(s)A_n]+\sum_{i=1}^nD^2f(we^{-s}+\sigma(s)A_n)[\mathbbm{1}_{[i/n,1]}^{(2)}]\right|\\
\leq&\|f\|_M\left(1+\|we^{-s+h}+\sigma(s-h)A_n\|^2\right)\|we^{-s+h}+\sigma(s-h)A_n\|\\
&+n\|f\|_M\left(1+\|we^{-s+h}+\sigma(s-h)A_n\|\right)\\
&+\|f\|_M\left(1+\|we^{-s}+\sigma(s)A_n\|\right)\|we^{-s}+\sigma(s)A_n\|+\|f\|_Mn\left(1+\|we^{-s}+\sigma(s)A_n\|\right)\\
\leq&\|f\|_M\left(1+2\|w\|^2e^{-2s+2}+2\sigma^2(s-1)\|A_n\|^2\right)\left(\|w\|e^{-s+1}+\sigma(s-1)\|A_n\|\right)\\
&+n\|f\|_M\left(1+\|we^{-s+1}+\sigma(s-1)A_n\|\right)\\
&+\|f\|_M\left(1+\|we^{-s}+\sigma(s)A_n\|\right)\|we^{-s}+\sigma(s)A_n\|+n\|f\|_M\left(1+\|we^{-s}+\sigma(s)A_n\|\right)
\end{align*}
for all $h\in[0,1]$.

Thus, for all $w\in D[0,1]$ and $s>0$: $\frac{d}{ds}T_{n,s}f(w)=T_{n,s}\mathcal{A}_nf(w)$ and so, by the Fundamental Theorem of Calculus:
\begin{equation}\label{4.12}
T_{n,r}f(w)-f(w)=\int_0^rT_{n,s}\mathcal{A}_nf(w)ds.
\end{equation}
Applying (\ref{4.12}) to $f=\int_0^tT_{n,u}gdu$ we obtain:
$$T_{n,r}\int_0^tT_{n,u}g(w)du-\int_0^tT_{n,u}g(w)du=\int_0^rT_{n,s}\mathcal{A}_n\left(\int_0^tT_{n,u}g(w)du\right)ds$$
Now, we take $t\to\infty$. We apply dominated convergence, which is allowed because of (\ref{4.77}) and the following bound for $h_{t,u}(w)=\int_0^tT_{n,u}g(w)du$:
\begin{align*}
&|\mathcal{A}_nh_{t,u}(w)|\\
\leq&\int_0^t\mathbbm{E}\left|e^{-u}Dg(we^{-u}+\sigma(u)A_n)[w]\right|du\\
&+\sum_{i=1}^n\int_0^t\mathbbm{E}\left|e^{-2u}D^2g(we^{-u}+\sigma(u)A_n)\left[\mathbbm{1}_{[i/n,1]}^{(2)}\right]\right|du\\
\leq&\int_0^\infty\mathbbm{E}\left|e^{-u}Dg(we^{-u}+\sigma(u)A_n)[w]\right|du\\
&+\sum_{i=1}^n\int_0^\infty\mathbbm{E}\left|e^{-2u}D^2g(we^{-u}+\sigma(u)A_n)\left[\mathbbm{1}_{[i/n,1]}^{(2)}\right]\right|du\\
\stackrel{(\ref{4.78})\text{A,B}}\leq&\left(1+\frac{4}{3}\mathbbm{E}\|A_n\|^2\right)\|g\|_M(1+\|w\|^2)\|w\|+n\left(\frac{1}{2}+\frac{\mathbbm{E}\|A_n\|}{3}\right)\|g\|_M(1+\|w\|),
\end{align*}
where the first inequality follows again by dominated convergence applied because of (\ref{4.77}) in order to exchange integration and differentiation in a way similar to (\ref{4.800}). As a result, we obtain:
\begin{align*}
T_{n,r}\int_0^\infty T_{n,u}g(w)du-\int_0^\infty T_{n,u}g(w)du&=\int_0^r T_{n,s}\lim_{t\to\infty}\mathcal{A}_n\left(\int_0^tT_{n,u}g(w)du\right)ds\\
&\stackrel{(\ref{4.100})}=-\int_0^r T_{n,s}g(w)ds.
\end{align*}
Now, dividing both sides by $r$ and taking $r\to 0$, we obtain:
\begin{align*}
\mathcal{A}_n\left(\int_0^\infty T_{n,u}g(w)\right)&=-\lim_{r\to 0}\frac{1}{r}\int_0^r T_{n,s}g(w)ds\\
&=-\lim_{r\to 0}\left[\frac{1}{r}\int_0^r\mathbbm{E}g(we^{-s}+\sigma(s)A_n)ds\right]\\
&\stackrel{(\ref{4.200})}=-g(w),
\end{align*}
which finishes the proof.
\end{proof}
\begin{remark}\label{remark}
It is an easy consequence of Propositions \ref{approx_proposition} and \ref{prop2.10} that for $g\in M$: $$\mathcal{A}_n\phi_n(g)(w)=-D\phi_n(g)(w)[w]+\mathbbm{E}D^2\phi_n(g)(w)\left[A_n^{(2)}\right].$$
\end{remark}
\section{Proofs of Theorems \ref{time_change_theorem}, \ref{theorem10.1} and \ref{theorem_moran}}\label{section4}
\subsection{Proof of Theorem \ref{time_change_theorem}}
\subsubsection{Discretisation of Brownian Motion}
Let $A_n$ be as in (\ref{a_n}).
Now, note that we can first realise $B$ and then set $A_n(t)=B\left(\frac{\lfloor ns(t)\rfloor}{n}\right)$ for $t\in[0,1]$ so that:
\begin{align*}
\sup_{t\in[0,1]}|A_n(t)-Z(t)|&=\sup_{t\in[0,1]}\left|B\left(\frac{\lfloor ns(t)\rfloor}{n}\right)-B(s(t))\right|=\sup_{t\in[0,s(1)]}\left|B(t)-B\left(\frac{\lfloor nt\rfloor}{n}\right)\right|.
\end{align*}
By Lemma 3 of \cite{ito_processes} we get:
\begin{align}
\text{A)}\quad\mathbbm{E}\|A_n-Z\|&\leq \mathbbm{E}\left[\sup_{t,s\in[0,s(1)],|t-s|\leq \frac{1}{n}}|B(t)-B(s)|\right]\nonumber\\
&\leq \frac{5}{\sqrt{\pi}}\cdot\frac{6}{\sqrt{\log 2}}n^{-1/2}\sqrt{\log (2s(1)n)}\nonumber\\
\text{B)}\quad\mathbbm{E}\|A_n-Z\|^2&\leq\mathbbm{E}\left[\left(\sup_{t,s\in[0,s(1)],|t-s|\leq \frac{1}{n}}|B(t)-B(s)|\right)^2\right]\nonumber\\
&\leq\frac{5}{2}\cdot\left(\frac{6}{\sqrt{\log 2}}\right)^2n^{-1}\log (2s(1)n)\nonumber\\
\text{C)}\quad\mathbbm{E}\|A_n-Z\|^3&\leq\mathbbm{E}\left[\left(\sup_{t,s\in[0,s(1)],|t-s|\leq \frac{1}{n}}|B(t)-B(s)|\right)^3\right]\nonumber\\
&\leq\frac{5}{\sqrt{\pi}}\cdot\left(\frac{6}{\sqrt{\log 2}}\right)^3n^{-3/2}(\log (2s(1)n))^{3/2}\label{10.1}
\end{align}
and therefore we obtain, for any $g\in M$:
\begin{align}
&|\mathbbm{E}g(A_n)-\mathbbm{E}g(Z)|\nonumber\\
\stackrel{\text{MVT}}\leq &\mathbbm{E}\left[\sup_{c\in[0,1]}\|Dg\left((1-c)Z+cA_n\right)\|\|Z-A_n\|\right]\nonumber\\
\leq&\|g\|_M\mathbbm{E}\left[\sup_{c\in[0,1]}\left(1+\|Z+c(A_n-Z)\|^2\right)\|Z-A_n\|\right] \nonumber\\
\leq& \|g\|_M\mathbbm{E}\left[\left(1+2\|Z\|^2+2\|Z-A_n\|^2\right)\|Z-A_n\|\right]\nonumber\\
\stackrel{\text{H\"older}}\leq&\|g\|_M\left\lbrace\mathbbm{E}\|Z-A_n\|+2\mathbbm{E}\|Z-A_n\|^3+2\left(\mathbbm{E}\|Z\|^3\right)^{2/3}\left(\mathbbm{E}\|A_n-Z\|^3\right)^{1/3}\right\rbrace \nonumber\\
\stackrel{(\ref{10.1})}\leq& \|g\|_M\left\lbrace\frac{5}{\sqrt{\pi}}\cdot\frac{6}{\sqrt{\log 2}}n^{-1/2}\sqrt{\log (2s(1)n)}\right.\nonumber\\
&+\frac{10}{\sqrt{\pi}}\cdot\left(\frac{6}{\sqrt{\log 2}}\right)^3n^{-3/2}(\log (2s(1)n))^{3/2}\nonumber\\
&\left.+2\left(\mathbbm{E}\|Z\|^3\right)^{2/3}\left(\frac{5}{\sqrt{\pi}}\right)^{1/3}\cdot\frac{6}{\sqrt{\log 2}}n^{-1/2}\sqrt{\log (2s(1)n)}\right\rbrace\nonumber\\
\stackrel{\text{Doob's }L^3}\leq& \|g\|_M\left(\frac{30}{\sqrt{\pi\log 2}}+\frac{2\cdot 5^{1/3}\cdot 6}{\pi^{1/6}\sqrt{\log 2}}\left(\left(\frac{3}{2}\right)^3\cdot 2\sqrt{\frac{2}{\pi}}\right)^{2/3}s(1)\right)n^{-1/2}\sqrt{\log (2s(1)n)}\nonumber\\
&+\|g\|_M\frac{10}{\sqrt{\pi}}\cdot\left(\frac{6}{\sqrt{\log 2}}\right)^3n^{-3/2}(\log (2s(1)n))^{3/2}\nonumber\\
=&\|g\|_M\frac{30+54\cdot 5^{1/3}s(1)}{\sqrt{\pi\log 2}}n^{-1/2}\sqrt{\log (2s(1)n)}\nonumber\\
&+\|g\|_M\frac{2160}{\sqrt{\pi}(\log 2)^{3/2}}n^{-3/2}(\log (2s(1)n))^{3/2}.
\label{10.2}
\end{align}

\subsubsection{Applying Stein's method}
Let $g\in M$ and $g_n=g-\mathbbm{E}[g(A_n)]$. Let $f_n=\phi_n(g_n)$, as in (\ref{phi}). First, note that:
\begin{equation*}
\begin{aligned}
\mathbbm{E}Df_n(Y_n)[Y_n]&=n^{-1/2}\sum_{j=1}^{\lfloor ns(1)\rfloor}\mathbbm{E}Df_n(Y_n)[X_j\mathbbm{1}_{[s^{-1}(j/n),1]}].
\end{aligned}
\end{equation*}
We now let $Y_n^j=n^{-1/2}\sum_{k\neq j} X_k\mathbbm{1}_{[s^{-1}(k/n),1]}=Y_n-n^{-1/2}X_j\mathbbm{1}_{[s^{-1}(j/n),1]}$ and observe that, by Taylor's theorem:
\begin{align}
&\left| n^{-1/2}\mathbbm{E}X_jDf_n(Y_n)\left[\mathbbm{1}_{[s^{-1}(j/n),1]}\right]\right.-\mathbbm{E}\left\lbrace n^{-1/2}X_jDf_n(Y_n^j)\left[\mathbbm{1}_{[s^{-1}(j/n),1]}\right]\right.\nonumber\\
&\left.\left.+n^{-1}\left(X_j\right)^2D^2f_n(Y_n^j)\left[\left(\mathbbm{1}_{[s^{-1}(j/n),1]}e_i\right)^{(2)}\right]\right\rbrace\right|\nonumber\\
=&\left|\mathbbm{E}\left[ n^{-1/2}X_jDf_n\left(Y_n^j+n^{-1/2}X_j\mathbbm{1}_{[j/n,1]}\right)\left[\mathbbm{1}_{[s^{-1}(j/n),1]}\right]\right.\right.\nonumber\\
&- n^{-1/2}X_jDf_n(Y_n^j)\left[\mathbbm{1}_{[s^{-1}(j/n),1]}\right]\nonumber\\
&\left.\left. -n^{-1}(X_j)^2D^2f_n(Y_n^j)\left[\left(\mathbbm{1}_{[s^{-1}(j/n),1]}\right)^{(2)}\right]\right]\right|\nonumber
\\\stackrel{(\ref{norm_bound})\text{C)}}\leq& \frac{n^{-3/2}}{6}\|g_n\|_{M}\mathbbm{E}|X_j|^3\label{this}
\end{align}
because, clearly, $\|\mathbbm{1}_{[s^{-1}(j/n),1]}\|=1$. Also, in the last inequality we have used the fact that $X_j$ is independent of $Y_n^j$. We can now sum (\ref{this}) over $j=1,2,...,\lfloor ns(1)\rfloor$ and use the fact that $X_j$'s are independent of $Y_n^j$'s and that $X_j$'s have mean $0$ and variance $1$ to obtain:
$$\left|\mathbbm{E}Df_n(Y_n)[Y_n]-n^{-1}\sum_{j=1}^{\lfloor ns(1)\rfloor}D^2f_n(Y_n^j)\left[\left(\mathbbm{1}_{[s^{-1}(j/n),1]}\right)^{(2)}\right]\right|\leq\frac{n^{-1/2}}{6}s(1)\|g_n\|_{M}\mathbbm{E}\left|X_1\right|^3\text{.}$$
We notice that for $\mathcal{A}_n$ defined in Proposition \ref{approx_proposition}, using Remark \ref{remark}, we obtain:

\begin{align*}
&\left|\mathbbm{E}\mathcal{A}_nf_n(Y_n)\right|=\left|\mathbbm{E}Df_n(Y_n)[Y_n]-\mathbbm{E}D^2f_n(Y_n)\left[A_n^{(2)}\right]\right|\\
\leq &\left|\mathbbm{E}Df_n(Y_n)[Y_n]-n^{-1}\sum_{j=1}^{\lfloor ns(1)\rfloor}\mathbbm{E}D^2f_n(Y^j_n)\left[\left(\mathbbm{1}_{[s^{-1}(j/n),1]}\right)^{(2)}\right]\right|\\
&+n^{-1}\left|\sum_{j=1}^{\lfloor ns(1)\rfloor}\mathbbm{E}\left\lbrace D^2f_n(Y_n)\left[\left(\mathbbm{1}_{[s^{-1}(j/n),1]}\right)^{(2)}\right]-D^2f_n(Y_n^j)\left[\left(\mathbbm{1}_{[s^{-1}(j/n),1]}\right)^{(2)}\right]\right\rbrace\right|\\
\leq& \frac{n^{-1/2}}{6}\|g_n\|_{M}\mathbbm{E}|X_1|^3\\
&+n^{-1}\left|\sum_{j=1}^{\lfloor ns(1)\rfloor}\mathbbm{E}\left\lbrace D^2f_n\left(Y_n^j+n^{-1/2}X_j\mathbbm{1}_{[s^{-1}(j/n),1]}\right)\left[\left(\mathbbm{1}_{[s^{-1}(j/n),1]}\right)^{(2)}\right]\right.\right.\\
&-\left.\left.D^2f_n(Y_n^j)\left[\left(\mathbbm{1}_{[s^{-1}(j/n),1]}\right)^{(2)}\right]\right\rbrace\right|\\
\stackrel{(\ref{norm_bound})C)}\leq&\frac{n^{-1/2}}{6}\|g_n\|_{M}\mathbbm{E}|X_1|^3+n^{-1}\frac{\|g_n\|_{M}}{3}\sum_{j=1}^{\lfloor ns(1)\rfloor}n^{-1/2}\mathbbm{E}\left\|X_j\mathbbm{1}_{[s^{-1}(j/n),1]}\right\|\\
\leq&\frac{n^{-1/2}}{6}s(1)\|g_n\|_{M}\mathbbm{E}|X_1|^3+\frac{n^{-1/2}}{3}s(1)\|g_n\|_{M}\mathbbm{E}\left\|X_1\mathbbm{1}_{[s^{-1}(j/n),1]}\right\|\\
\leq&\frac{n^{-1/2}s(1)\|g_n\|_{M}}{2}\mathbbm{E}|X_1|^3\text{.}
\end{align*}
The last inequality follows by Jensen's inequality:
\begin{equation}\label{jensens_ineq}
\mathbbm{E}|X_1|\leq\sqrt{\mathbbm{E}|X_1|^2}=1=\left(\mathbbm{E}|X_1|^2\right)^{3/2}\leq \mathbbm{E}|X_1|^3\text{.}
\end{equation}

Now, note that this gives:
\begin{align}
&|\mathbbm{E}g(Y_n)-\mathbbm{E}g(A_n)|=|\mathbbm{E}g_n(Y_n)|=|\mathbbm{E}\mathcal{A}_nf_n(Y_n)|\nonumber\\
\leq &  \frac{n^{-1/2}}{2}s(1)\|g_n\|_{M}\mathbbm{E}|X_1|^3\nonumber\\
\leq &\frac{n^{-1/2}}{2}s(1)\left(\|g\|_{M}+\mathbbm{E}g(A_n)\right)\mathbbm{E}|X_1|^3\nonumber\\
\leq &\frac{n^{-1/2}}{2}s(1)\left(2+\mathbbm{E}\|A_n\|^3\right)\|g\|_{M}\mathbbm{E}|X_1|^3\text{.}\label{proof0.5}
\end{align}
Also, recall that, by Doob's $L^p$ inequality (see, e.g. Theorem 1.7 of Chapter II in \cite{revuz}), if $M$ is a right-continuous martingale then for every $t>0$ and $p>1$:
$$\mathbbm{E}\left[\left(\sup_{s\in[0,t]}|M_s|\right)^p\right]\leq \left(\frac{p}{p-1}\right)^p\mathbbm{E}|M_t|^p.$$
Note that $A_n$ is an integrable process, adapted to its natural filtration and for any $t\in [s^{-1}(m/n),s^{-1}((m+1)/n))$ and $r\in [s^{-1}(l/n),s^{-1}((l+1)/n))$ and $r<t$:
\begin{align*}
\mathbbm{E}\left[A_n(t)\left|\left\lbrace A_n(u):u\leq r\right\rbrace\right.\right]&=n^{-1/2}\mathbbm{E}\left[\left.\sum_{k=1}^mZ_k\right|Z_1,...,Z_l\right]\\
&=n^{-1/2}\sum_{k=1}^lZ_k + n^{-1/2}\mathbbm{E}\left[\sum_{k=l+1}^m Z_k\right]\\
& = A_n(r)
\end{align*}
and so $A_n$ is a right-continuous martingale. Applying Doob's $L^3$ inequality to it yields:
\begin{equation*}
\begin{aligned}
&\mathbbm{E}\|A_n\|^3\leq \left(\frac{3}{2}\right)^3\mathbbm{E}|A_n(1)|^3=\left(\frac{3}{2}\right)^3\cdot 2\sqrt{\frac{2}{\pi}}s(1)^{3/2}\text{,}
\end{aligned}
\end{equation*}
because $A_n(1)\sim\mathcal{N}(0,\frac{\lfloor ns(1)\rfloor}{n})$.
Therefore (\ref{proof0.5}) gives:
\begin{equation}\label{proof0.7}
|\mathbbm{E}g(Y_n)-\mathbbm{E}g(A_n)|\leq n^{-1/2}\left(1+\left(\frac{3}{2}\right)^3\cdot\sqrt{\frac{2}{\pi}}s(1)^{3/2}\right)\|g\|_{M}s(1)\mathbbm{E}|X_1|^3\text{.}
\end{equation}

Combining this with (\ref{10.2}):
\begin{align}
|\mathbbm{E}g(Y_n)-\mathbbm{E}g(Z)|\leq& \|g\|_M\frac{30+54\cdot 5^{1/3}s(1)}{\sqrt{\pi\log 2}}n^{-1/2}\sqrt{\log (2s(1)n)}\nonumber\\
&+\|g\|_Ms(1)\left(1+\left(\frac{3}{2}\right)^3\sqrt{\frac{2}{\pi}}s(1)^{3/2}\right)\mathbbm{E}|X_1|^3n^{-1/2}\nonumber\\
&+\|g\|_M\frac{2160}{\sqrt{\pi}(\log 2)^{3/2}}n^{-3/2}(\log (2s(1)n))^{3/2},
\label{10.4}
\end{align}
which proves Theorem \ref{time_change_theorem}.\qed
\subsection{Proof of Theorem \ref{theorem10.1}}\label{section4}
Note that $\left(P(\lfloor nS^{(n)}(t)\rfloor),t\in[0,1]\right)$ can be expressed in the following way:
$$P(\lfloor nS^{(n)}(t)\rfloor)-\lfloor nS^{(n)}(t)\rfloor =\sum_{i=1}^{\lfloor nS^{(n)}(t)\rfloor} X_i,$$
where $(X_i+1)$'s are i.i.d. Poisson($1$). 
Therefore, we can express $(\tilde{Y}_n(t),t\in[0,1])$ in the following way:
\begin{align*}
\tilde{Y}_n(t)=&n^{-1/2}\left\lbrace\sum_{i=1}^{\lfloor nS^{(n)}(t)\rfloor} X_i+P\left(nS^{(n)}(t)\right)-P\left(\lfloor nS^{(n)}(t)\rfloor\right)-\left(nS^{(n)}(t)-\lfloor nS^{(n)}(t)\rfloor\right)\right\rbrace.
\end{align*}
We also define:
$$Y_n(t)=n^{-1/2}\sum_{i=1}^{\lfloor nS^{(n)}(t)\rfloor} X_i.$$
Note that $\left|nS^{(n)}(t)-\lfloor nS^{(n)}(t)\rfloor\right|\leq 1$ for all $t\geq 0$. Also, observe that for all $t\geq 0$: 
$$\left|P\left(nS^{(n)}(t)\right)-P\left(\lfloor nS^{(n)}(t)\rfloor\right)\right|\leq P\left(\lfloor nS^{(n)}(t)\rfloor+1\right)-P\left(\lfloor nS^{(n)}(t)\rfloor\right).$$
By the independence of increments of a Poisson process:
\begin{align}
&\text{A)} \quad\mathbbm{E}\|\tilde{Y}_n-Y_n\|\leq n^{-1/2}\left[1+\mathbbm{E}\left[\max_{1\leq i\leq n}\bar{P}_i\right]\right]\nonumber\\
&\text{B)} \quad\mathbbm{E}\|\tilde{Y}_n-Y_n\|^3\leq n^{-3/2}\left[4+4\mathbbm{E}\left[\max_{1\leq i\leq n}\bar{P}_i^3\right]\right],\label{5.87}
\end{align}
where $\bar{P}_1,\cdots,\bar{P}_n\stackrel{\text{i.i.d}}\sim \textbf{Poisson}(1)$. Using the trick from \cite{simple_trick}, we note that, by Jensen's inequality applied to function $\exp(x\log\log (n+2))$:
\begin{align}
\exp\left(\log\log (n+2)\cdot\mathbbm{E}\left[\max_{1\leq i\leq n}\bar{P}_i\right]\right)&\leq \mathbbm{E}\left[\exp\left(\log\log (n+2)\cdot \max_{1\leq i\leq n}\bar{P}_i\right)\right]\nonumber\\
&=\mathbbm{E}\left[\max_{1\leq i\leq n}\exp\left(\log\log (n+2)\cdot\bar{P}_i\right)\right]\nonumber\\
&\leq n \mathbbm{E}\left[\exp\left(\log\log n\cdot\bar{P}_1\right)\right]\nonumber\\
&=n\exp\left(\log (n+2)-1\right)\nonumber\\
&\leq \left(1+2e^{-1}\right)n^2\label{5.88}
\end{align}
and by Jensen's inequality applied to function $\exp(x^{1/3}\log\log (n+3))$, which is convex for $x\geq \frac{8}{(\log\log (n+3))^3}$:

\begin{align}
&\exp\left(\log\log (n+3)\cdot\left\lbrace\mathbbm{E}\left[\left.\max_{1\leq i\leq n}\bar{P}_i^3\right|\left(\max_{1\leq i\leq n}\bar{P}_i^3\right)\geq \frac{8}{\left(\log\log (n+3)\right)^3}\right]\right\rbrace^{1/3}\right)\nonumber\\
\leq&\mathbbm{E}\left[\left.\exp\left(\log\log (n+3)\cdot\max_{1\leq i\leq n}\bar{P}_i\right)\right|\left(\max_{1\leq i\leq n}\bar{P}_i^3\right)\geq \frac{8}{\left(\log\log (n+3)\right)^3}\right]\nonumber\\
\leq&n\mathbbm{E}\left[\exp\left(\log\log (n+3)\cdot\bar{P}_1\right)\left|\left(\max_{1\leq i\leq n}\bar{P}_i^3\right)\geq \frac{8}{\left(\log\log (n+3)\right)^3}\right.\right]\nonumber\\
\leq& \frac{n\exp(\log (n+3)-1)}{\mathbbm{P}\left[\left(\max_{1\leq i\leq n}\bar{P}_i\right)\geq \frac{2}{\left(\log\log (n+3)\right)}\right]}\nonumber\\
\leq&\frac{n\exp(\log (n+3)-1)}{\mathbbm{P}\left[\bar{P}_1\geq \frac{2}{\left(\log\log 4)\right)}\right]}\nonumber\\
\leq&\frac{(1+3e^{-1})n^2}{1-\frac{1957}{720e}}.
\label{5.89}
\end{align}
Now, combining (\ref{5.87}), (\ref{5.88}) and (\ref{5.89}), we obtain:
\begin{align}
\text{A)}\quad &\mathbbm{E}\|\tilde{Y}_n-Y_n\|\leq n^{-1/2}\left[1+\frac{\log(1+2e^{-1})+2\log n}{\log\log(n+2)}\right]\nonumber\\
\text{B)}\quad&\mathbbm{E}\|\tilde{Y}_n-Y_n\|^3\leq n^{-3/2}\left[4+4\left(\frac{\log\left(\frac{1+3e^{-1}}{1-\frac{1957}{720e}}\right)+2\log n}{\log\log(n+3)}\right)^3+\frac{32}{(\log\log(n+3))^3}\right]\nonumber\\
&\leq n^{-3/2}\left[4+\frac{16701+128(\log n)^3}{(\log\log(n+3))^3}\right].
\label{5.90}
\end{align}
We also note that:
\begin{align}
\left[\mathbbm{E}\left\|\tilde{Y}_n\right\|^3\right]^{2/3}&\stackrel{\text{Doob}}\leq n^{-1}\left(\frac{3}{2}\right)^2\left[\mathbbm{E}\left|P\left(nS^{(n)}(1)\right)-nS^{(n)}(1)\right|^3\right]^{2/3}\nonumber\\
&\leq\left(\frac{3}{2}\right)^2n^{-1}\left[\mathbbm{E}\left|P\left(nS^{(n)}(1)\right)-nS^{(n)}(1)\right|^4\right]^{1/2}\nonumber\\
&=\left(\frac{3}{2}\right)^2n^{-1/2}\sqrt{S^{(n)}(1)}\left(1+3nS^{(n)}(1)\right)^{1/2}.
\label{10.5}
\end{align}

Then, for every $g\in M$:
\begin{align}
&|\mathbbm{E}g(Y_n)-\mathbbm{E}g(\tilde{Y}_n)|\stackrel{\text{MVT}}\leq \mathbbm{E}\left[\sup_{c\in[0,1]}\|Dg((1-c)\tilde{Y}_n+cY_n)\|\|Y_n-\tilde{Y}_n\|\right]\nonumber \\
\leq&\|g\|_M\mathbbm{E}\left[\sup_{c\in[0,1]}(1+\|\tilde{Y}_n+c(Y_n-\tilde{Y}_n)\|^2)\|Y_n-\tilde{Y}_n\|\right]\nonumber\\
\leq& \|g\|_M\mathbbm{E}\left[(1+2\|\tilde{Y}_n\|^2+2\|Y_n-\tilde{Y}_n\|^2)\|Y_n-\tilde{Y}_n\|\right]\nonumber\\
\stackrel{\text{H\"older}}\leq& \|g\|_M\left\lbrace\mathbbm{E}\|Y_n-\tilde{Y}_n\|+2\mathbbm{E}\|Y_n-\tilde{Y}_n\|^3+2(\mathbbm{E}\|\tilde{Y}_n\|^3)^{2/3}(\mathbbm{E}\|Y_n-\tilde{Y}_n\|^3)^{1/3}\right\rbrace\nonumber\\
\stackrel{(\ref{5.90}),(\ref{10.5})}\leq& \|g\|_M\left\lbrace n^{-1/2}\left[1+\frac{\log(1+2e^{-1})+2\log n}{\log\log(n+2)}\right]+n^{-3/2}\left[8+\frac{33402+256(\log n)^3}{(\log\log(n+3))^3}\right]\right.\nonumber\\
&\left.+\frac{9}{2}n^{-1}\sqrt{S^{(n)}(1)}\left(1+3nS^{(n)}(1)\right)^{1/2}\left[4+\frac{16701+128(\log n)^3}{(\log\log(n+3))^3}\right]^{1/3}\right\rbrace.
\label{10.6}
\end{align}
Let $A_n(t)=n^{-1/2}\sum_{i=1}^{\lfloor nS^{(n)}(t)\rfloor} Z_i,t\in[0,1]$ for $Z_i\stackrel{i.i.d}\sim\mathcal{N}(0,1)$. By (\ref{proof0.7}):
\begin{align}
|\mathbbm{E}g(Y_n)-\mathbbm{E}g(A_n)|&\leq n^{-1/2}\left(1+\left(\frac{3}{2}\right)^3\sqrt{\frac{2}{\pi}}S^{(n)}(1)^{3/2}\right)S^{(n)}(1)\|g\|_M\mathbbm{E}|X_1|^3\nonumber\\
&\leq n^{-1/2}\left(1+\left(\frac{3}{2}\right)^3\sqrt{\frac{2}{\pi}}S^{(n)}(1)^{3/2}\right)S^{(n)}(1)\|g\|_M(1+2e^{-1})\label{10.7}
\end{align}
because $X_1\stackrel{\mathcal{D}}=P(1)-1$.
Now let $\tilde{A}_n(t)=n^{-1/2}\sum_{i=1}^{\lfloor nS(t)\rfloor} Z_i,t\in[0,1]$. Then:
\begin{align}
&\text{A)} \quad\mathbbm{E}\|A_n-\tilde{A}_n\|=n^{-1/2}\mathbbm{E}\left[\sup_{t\in[0,1]}\left|\sum_{i=\lfloor nS(t)\wedge S^{(n)}(t)\rfloor+1}^{\lfloor nS(t)\vee S^{(n)}(t)\rfloor} Z_i\right|\right]\nonumber\\
&=n^{-1/2}\mathbbm{E}\left[\sup_{t\in[0,1]}\left|\sum_{i=1}^{\lfloor nS(t)\vee S^{(n)}(t)\rfloor-(\lfloor nS(t)\wedge S^{(n)}(t)\rfloor+1)} Z_i\right|\right]\nonumber\\
&\stackrel{Doob,Jensen}\leq 2n^{-1/2}\sqrt{\mathbbm{E}\left|\sum_{i=1}^{\sup_{t\in[0,1]}(\lfloor nS(t)\vee S^{(n)}(t)\rfloor-(\lfloor nS(t)\wedge S^{(n)}(t)\rfloor+1))} Z_i\right|^2}\leq 2\sqrt{\|S-S^{(n)}\|}\nonumber\\
&\text{B)} \quad\mathbbm{E}\|A_n-\tilde{A}_n\|^3=n^{-3/2}\mathbbm{E}\left[\sup_{t\in[0,1]}\left|\sum_{i=\lfloor nS(t)\wedge S^{(n)}(t)\rfloor+1}^{\lfloor nS(t)\vee S^{(n)}(t)\rfloor} Z_i\right|^3\right]\nonumber\\
&=n^{-3/2}\mathbbm{E}\left[\sup_{t\in[0,1]}\left|\sum_{i=1}^{\lfloor nS(t)\vee S^{(n)}(t)\rfloor-(\lfloor nS(t)\wedge S^{(n)}(t)\rfloor+1)} Z_i\right|^3\right]\nonumber\\
&\stackrel{Doob}\leq \left(\frac{3}{2}\right)^3n^{-3/2}\mathbbm{E}\left[\left|\sum_{i=1}^{\sup_{t\in[0,1]}(\lfloor nS(t)\vee S^{(n)}(t)\rfloor-(\lfloor nS(t)\wedge S^{(n)}(t)\rfloor+1))} Z_i\right|^3\right]\nonumber\\
&\leq 2\sqrt{\frac{2}{\pi}}\left(\frac{3}{2}\right)^3\|S-S^{(n)}\|^{3/2}\nonumber\\
&\text{C)}\quad \left(\mathbbm{E}\|\tilde{A}_n\|^3\right)^{2/3}\stackrel{\text{Doob}}\leq 2n^{-1}\left(\frac{3}{2}\right)^2\pi^{-1/3}\left(S(1)^{3/2}\right)^{2/3}=\frac{9}{2\pi^{1/3}}n^{-1}S(1).
\label{11.91}
\end{align}
Therefore:
\begin{align}
&|\mathbbm{E}g(A_n)-\mathbbm{E}g(\tilde{A_n})|\stackrel{\text{MVT}}\leq \mathbbm{E}\left[\sup_{c\in[0,1]}\|Dg\left((1-c)\tilde{A}_n+cA_n\right)\|\|A_n-\tilde{A}_n\|\right]\nonumber\\
\leq&\|g\|_M\mathbbm{E}\left[\sup_{c\in[0,1]}(1+\|\tilde{A}_n+c(A_n-\tilde{A}_n)\|^2)\|A_n-\tilde{A}_n\|\right]\nonumber\\
\leq& \|g\|_M\mathbbm{E}\left[(1+2\|\tilde{A}_n\|^2+2\|A_n-\tilde{A}_n\|^2)\|A_n-\tilde{A_n}\|\right]\nonumber\\
\stackrel{\text{H\"older}}\leq &\|g\|_M\left\lbrace \mathbbm{E}\|A_n-\tilde{A}_n\|+2\mathbbm{E}\|A_n-\tilde{A}_n\|^3+2(\mathbbm{E}\|\tilde{A}_n\|^3)^{2/3}\left(\mathbbm{E}\|A_n-\tilde{A}_n\|\right)^{1/3}\right\rbrace\nonumber\\
\stackrel{(\ref{11.91})}\leq&\|g\|_M\left\lbrace 2\sqrt{\|S-S^{(n)}\|}+\frac{27\sqrt{2}}{2\sqrt{\pi}}\|S-S^{(n)}\|^{3/2}+\frac{27\sqrt{2}}{2\sqrt{\pi}}S(1)\sqrt{\|S-S^{(n)}\|}\right\rbrace.
\label{10.8}
\end{align}

By (\ref{10.2}) we get for $Z=B\circ S$:
\begin{equation}\label{10.9}
\begin{aligned}
&|\mathbbm{E}g(\tilde{A}_n)-\mathbbm{E}g(Z)|\leq
\|g\|_M\frac{30+54\cdot 5^{1/3}S(1)}{\sqrt{\pi\log 2}}n^{-1/2}\sqrt{\log (2S(1)n)}\\
&+\|g\|_M\frac{2160}{\sqrt{\pi}(\log 2)^{3/2}}n^{-3/2}(\log (2S(1)n))^{3/2}.
\end{aligned}
\end{equation}
Theorem \ref{theorem10.1} now follows from (\ref{10.6}), (\ref{10.7}), (\ref{10.8}), (\ref{10.9}).\qed
\subsection{Proof of Theorem \ref{theorem_moran}}\label{section5}
Note that $X_n$ jumps up by $\frac{1}{n}$ with intensity $\frac{1}{2}n^2X_n(t)(1-X_n(t))+n\nu_2(1-X_n(t))$ and down by $\frac{1}{n}$ with intensity $\frac{1}{2}n^2X_n(t)(1-X_n(t))+n\nu_1X_n(t)$. To see this observe that a jump occurs with intensity ${n\choose 2}$ and it is an up-jump if the first gene chosen was of type a, the second of type A and the one with type A died (which happens with probability $\frac{1}{2}X_n(t)\frac{n(1-X_n(t))}{n-1}$) or if the first one chosen was of type A, the second of type a and the type A gene died (which happens with probability $\frac{1}{2}(1-X_n(t))\frac{nX_n(t)}{n-1}$). In addition, there are $n(1-X_n(t))$ genes of type A and each of them mutates into type a at rate $\nu_2$. Hence:
\begin{align*}
\begin{cases}
\mathbbm{P}\left[X_n(t+\Delta t)-X_n(t)=\frac{1}{n}\right]=\frac{1}{2}n^2X_n(t)(1-X_n(t))\Delta t+n\nu_2(1-X_n(t))\Delta t\\
\mathbbm{P}\left[X_n(t+\Delta t)-X_n(t)=-\frac{1}{n}\right]=\frac{1}{2}n^2X_n(t)(1-X_n(t))\Delta t+n\nu_1X_n(t)\Delta t.
\end{cases}
\end{align*}
Therefore:
\begin{align*}
M_n(t)=&\frac{P_1\left(n^2R^{(n)}_1(t)\right)-n^2R^{(n)}_1(t)}{n}-\frac{P_{-1}\left(n^2R^{(n)}_{-1}(t)\right)-n^2R^{(n)}_{-1}(t)}{n}\\
&+\int_0^t\left(\nu_2-(\nu_1+\nu_2)M_n(s)\right)ds,
\end{align*}
where $P_1,P_{-1}$ are i.i.d. Poisson processes with rate 1, independent of $X_n$, and  
$$\begin{cases}
R_1^{(n)}(t):=\int_0^t \left(\frac{1}{2}X_n(s)+\frac{\nu_2}{n}\right)(1-X_n(s))ds\\
R_{-1}^{(n)}(t):=\int_0^t \left(\frac{1}{2}(1-X_n(s))+\frac{\nu_1}{n}\right)X_n(s)ds
\end{cases}\text{for } t\in[0,1].$$
Also let:
$$\begin{cases}
R_1(t)=R_{-1}(t):=\int_0^t\frac{1}{2}X(s)(1-X(s))ds\\
I_n(t):=\int_0^t\left(\nu_2-(\nu_1+\nu_2)X_n(s)\right)ds\\
I(t):=\int_0^t\left(\nu_2-(\nu_1+\nu_2)X(s)\right)ds
\end{cases} \text{for }t\in[0,1].$$

Let us denote $Z_1=B_1\circ R_1$, $Z_{-1}=B_{-1}\circ R_{-1}$, where $B_1$ and $B_{-1}$ are i.i.d. standard Brownian Motions, independent of $X$ and:
$$Y_n^1(\cdot):=\frac{P_1\left(n^2R_1^{(n)}(\cdot)\right)-n^2R_1^{(n)}(\cdot)}{n},\qquad Y_n^{-1}(\cdot):=\frac{P_{-1}\left(n^2R_{-1}^{(n)}(\cdot)\right)-n^2R_{-1}^{(n)}(\cdot)}{n}.$$ 

Now, for any $g\in M$:
\begin{align}
&|\mathbbm{E}g(M_n)-\mathbbm{E}g(M)|\nonumber\\
\leq& \left|\mathbbm{E}\left[\mathbbm{E}\left[g(Y_n^1-Y_n^{-1}+I_n)-g(Z_1-Y_n^{-1}+I_n)|Y_n^{-1},I_n,R_1^{(n)},R_{-1}^{(n)},R_1\right]\right]\right.\nonumber\\
&+\mathbbm{E}\left[\mathbbm{E}\left[g(Z_1-Y_n^{-1}+I_n)-g(Z_1-Z_{-1}+I_n)|Z_1,I_n,R_1^{(n)},R_{-1}^{(n)},R_{1}\right]\right]\nonumber\\
&\left.+\mathbbm{E}\left[\mathbbm{E}\left[g(Z_1-Z_{-1}+I_n)-g(Z_1-Z_{-1}+I)|Z_1,Z_{-1},R_{1}\right]\right]\right|\nonumber\\
=&\left|\mathbbm{E}\left[\mathbbm{E}\left[g^{(1)}(Y_n^1)-g^{(1)}(Z_1)|Y_n^{-1},I_n,R_1^{(n)},R_1\right]\right]\nonumber\right.\\
&+\mathbbm{E}\left[\mathbbm{E}\left[g^{(-1)}(Y_n^{-1})-g^{(-1)}(Z_{-1})|Z_1,I_n,R_{-1}^{(n)},R_{-1}\right]\right]\nonumber\\
&\left.+\mathbbm{E}\left[\mathbbm{E}\left[g^{(0)}(I_n)-g^{(0)}(I)|Z_1,Z_{-1},R_{1},R_{-1}\right]\right]\right|,
\label{11.29}
\end{align}
where $g^{(1)}(x)=g(x-Y_n^{-1}+I_n)$, $g^{(-1)}(x)=g(Z_1-x+I_n)$, $g^{(0)}(x)=g(Z_1-Z_{-1}+x)$. Note that, given $R_1$ and $R_1^{(n)}$, $Y_n^1$ and $Z_1$ are independent of $Y_n^{-1}$ and $I_n$. Similarly, given $R_{-1}$ and $R_{-1}^{(n)}$, $Y_n^{-1}$ and $Z_{-1}$ are independent of $Z_1$ and $I_n$. Also, given $R_1$ and $R_{-1}$, $I_n$ and $I$ are independent of $Z_1$ and $Z_{-1}$.

We now apply Theorem \ref{theorem10.1} to obtain:
\begin{align}
&\text{A)}\quad \mathbbm{E}\left[\mathbbm{E}\left[g^{(1)}(Y_n^1)-g^{(1)}(Z_1)|Y_n^{-1},I_n,R_1^{(n)},R_1\right]\right]\nonumber\\
\leq& \mathbbm{E}\left[\mathbbm{E}\left[\|g^{(1)}\|_M|R_1,R_1^{(n)}\right]\right.\nonumber\\
&\cdot\left\lbrace \left(2+\frac{27\sqrt{2}}{2\sqrt{\pi}}R_1(1)\right)\sqrt{\|R_1-R_1^{(n)}\|}+\frac{27\sqrt{2}}{2\sqrt{\pi}}\|R_1-R_1^{(n)}\|^{3/2}\right.\nonumber\\
&+ n^{-1}\left[\frac{30+54\cdot 5^{1/3}R_1(1)}{\sqrt{\pi\log 2}}\sqrt{\log (2R_1(1)n^2)}\right.+\left(1+\left(\frac{3}{2}\right)^3\sqrt{\frac{2}{\pi}}R_1^{(n)}(1)^{3/2}\right)\nonumber\\
&\cdot R_1^{(n)}(1)(1+2e^{-1})
\left.+1+\frac{\log(1+2e^{-1})+4\log n}{\log\log(n^2+2)}\right]\nonumber\\
&+n^{-2}\frac{9\sqrt{R_1^{(n)}(1)}}{2}\left(1+3n^2R_1^{(n)}(1)\right)^{1/2}\left[4+\frac{16701+1024(\log n)^3}{(\log\log(n^2+3))^3}\right]^{1/3}\nonumber\\
&+\left.n^{-3}\left[\frac{2160}{\sqrt{\pi}(\log 2)^{3/2}}(\log (2R_1(1)n^2))^{3/2}+8+\frac{33402+2048(\log n)^3}{(\log\log(n^2+3))^3}\right]\right\rbrace\nonumber\\
\leq&\mathbbm{E}\left[\mathbbm{E}\left[\|g^{(1)}\|_M|R_1,R_1^{(n)}\right]\left\lbrace \left(2+\frac{27\sqrt{2}}{16\sqrt{\pi}}\right)\sqrt{\|R_1-R_1^{(n)}\|}+\frac{27\sqrt{2}}{2\sqrt{\pi}}\|R_1-R_1^{(n)}\|^{3/2}\right.\right.\nonumber\\
&+ n^{-1}\left[\frac{30+\frac{27}{4}\cdot 5^{1/3}}{\sqrt{\pi\log 2}}\sqrt{\log (n^2/4)}\right.+\left(1+\left(\frac{3}{2}\right)^3\sqrt{\frac{2}{\pi}}\left(\frac{1}{8}+\frac{\nu_2}{n}\right)^{3/2}\right)\nonumber\\
&\cdot\left.\left(\frac{1}{8}+\frac{\nu_2}{n}\right)(1+2e^{-1})+1+\frac{\log(1+2e^{-1})+4\log n}{\log\log(n^2+2)}\right]\nonumber\\
&+n^{-2}\frac{9\sqrt{\frac{1}{8}+\frac{\nu_2}{n}}}{2}\left(1+\frac{3n^2}{8}+3n\nu_2\right)^{1/2}\left[4+\frac{16701+1024(\log n)^3}{(\log\log(n^2+3))^3}\right]^{1/3}\nonumber\\
&+\left.\left.n^{-3}\left[\frac{2160}{\sqrt{\pi}(\log 2)^{3/2}}(\log (n^2/4+\nu_2n))^{3/2}+8+\frac{33402+2048(\log n)^3}{(\log\log(n^2+3))^3}\right]\right\rbrace\right]\nonumber\\
&\text{B)}\quad\left|\mathbbm{E}\left(\mathbbm{E}\left\lbrace\left.\mathbbm{E}\left[\left.g^{(-1)}\left(Y_n^{-1}\right)-g^{(-1)}\left(Z_{-1}\right)\right|Z_1,I_n\right]\right|R_{-1},R_{-1}^{(n)}\right\rbrace\right)\right|\nonumber\\
\leq&\mathbbm{E}\left[\mathbbm{E}\left[\left.\|g^{(-1)}\|_M\right.|R_{-1},R_{-1}^{(n)}\right]\right.\nonumber\\
&\cdot\left\lbrace \left(2+\frac{27\sqrt{2}}{2\sqrt{\pi}}R_{-1}(1)\right)\sqrt{\|R_{-1}-R_{-1}^{(n)}\|}+\frac{27\sqrt{2}}{2\sqrt{\pi}}\|R_{-1}-R_{-1}^{(n)}\|^{3/2}\right.\nonumber\\
&+ n^{-1}\left[\frac{30+54\cdot 5^{1/3}R_{-1}(1)}{\sqrt{\pi\log 2}}\sqrt{\log (2R_{-1}(1)n^2)}+\left(1+\left(\frac{3}{2}\right)^3\sqrt{\frac{2}{\pi}}R_{-1}^{(n)}(1)^{3/2}\right)\nonumber\right.\\
&\cdot\left. R_{-1}^{(n)}(1)(1+2e^{-1})+1+\frac{\log(1+2e^{-1})+4\log n}{\log\log(n^2+2)}\right]\nonumber\\
&+n^{-2}\frac{9\sqrt{R_{-1}^{(n)}(1)}}{2}\left(1+3n^2R_{-1}^{(n)}(1)\right)^{1/2}\left[4+\frac{16701+1024(\log n)^3}{(\log\log(n^2+3))^3}\right]^{1/3}\nonumber\\
&+\left.n^{-3}\left[\frac{2160}{\sqrt{\pi}(\log 2)^{3/2}}(\log (2R_{-1}(1)n^2))^{3/2}+8+\frac{33402+2048(\log n)^3}{(\log\log(n^2+3))^3}\right]\right\rbrace\nonumber\\
\leq&\mathbbm{E}\left[\mathbbm{E}\left[\|g^{(1)}\|_M|R_{-1},R_{-1}^{(n)}\right]\left\lbrace \left(2+\frac{27\sqrt{2}}{16\sqrt{\pi}}\right)\sqrt{\|R_{-1}-R_{-1}^{(n)}\|}+\frac{27\sqrt{2}}{2\sqrt{\pi}}\|R_{-1}-R_{-1}^{(n)}\|^{3/2}\right.\right.\nonumber\\
&+ n^{-1}\left[\frac{30+\frac{27}{4}\cdot 5^{1/3}}{\sqrt{\pi\log 2}}\sqrt{\log (n^2/4)}\right.+\left(1+\left(\frac{3}{2}\right)^3\sqrt{\frac{2}{\pi}}\left(\frac{1}{8}+\frac{\nu_1}{n}\right)^{3/2}\right)\nonumber\\
&\left.\cdot\left(\frac{1}{8}+\frac{\nu_1}{n}\right)(1+2e^{-1})+1+\frac{\log(1+2e^{-1})+4\log n}{\log\log(n^2+2)}\right]\nonumber\\
&+n^{-2}\frac{9\sqrt{\frac{1}{8}+\frac{\nu_1}{n}}}{2}\left(1+\frac{3n^2}{8}+3n\nu_1\right)^{1/2}\left[4+\frac{16701+1024(\log n)^3}{(\log\log(n^2+3))^3}\right]^{1/3}\nonumber\\
&+\left.\left.n^{-3}\left[\frac{2160}{\sqrt{\pi}(\log 2)^{3/2}}(\log (n^2/4+\nu_1n))^{3/2}+8+\frac{33402+2048(\log n)^3}{(\log\log(n^2+3))^3}\right]\right\rbrace\right],
\label{11.31}
\end{align}
where we have used the fact that $R_1(1),R_{-1}(1)\leq \frac{1}{8}$ and $R_1^{(n)}\leq \frac{1}{8}+\frac{\nu_2}{n}$, $R_{-1}^{(n)}\leq \frac{1}{8}+\frac{\nu_1}{n}$. We also note that:
\begin{align}
&\left|\mathbbm{E}\left[\mathbbm{E}\left[g^{(0)}(I_n)-g^{(0)}(I)|Z_1,Z_{-1},R_1,R_{-1}\right]\right]\right|\nonumber\\
\leq&\mathbbm{E}\left[\mathbbm{E}\left[\|g^{(0)}\|_M|R_1,R_{-1}\right]\left[\|I_n-I\|+2\|I_n-I\|^3+2\|I\|\|I_n-I\|\right]\right]\nonumber\\
\leq&\mathbbm{E}\left[\mathbbm{E}\left[\|g^{(0)}\|_M|R_1,R_{-1}\right]\left[\|I_n-I\|+2\|I_n-I\|^3+2\nu_2\|I_n-I\|\right]\right].\label{11.32}
\end{align}
Now note that:
\begin{align*}
&\text{A)}\quad\mathbbm{E}\left[\left.\sup_{w,h\in D}\frac{\left\|D^2g\left(w+h-Y_n^{-1}+I_n\right)-D^2g\left(w-Y_n^{-1}+I_n\right)\right\|}{\|h\|}\right|R_1,R_1^{(n)},R_{-1}^{(n)}\right]\\
\leq&\sup_{x,y\in D}\frac{\|D^2g(x+y)-D^2g(x)\|}{\|y\|}\\
&\text{B)}\quad \mathbbm{E}\left[\left.\sup_{w\in D}\frac{\left\|D^2g\left(w-Y_n^{-1}+I_n\right)\right\|}{1+\|w\|}\right|R_1,R_1^{(n)},R_{-1}^{(n)}\right]\\
\leq&\mathbbm{E}\left[\left.\sup_{w\in D}\frac{\left\|D^2g\left(w-Y_n^{-1}+I_n\right)\right\|}{1+\left\|w-Y_n^{-1}+I_n\right\|}\cdot \frac{1+\left\|w-Y_n^{-1}+I_n\right\|}{1+\|w\|}\right|R_1,R_1^{(n)},R_{-1}^{(n)}\right]\\
\stackrel{Doob}\leq& \left(\sup_{w\in D}\frac{\|D^2g(w)\|}{1+\|w\|}\right)\left[1+2\sqrt{\mathbbm{E}\left[\left.\left(Y_n^{-1}(1)\right)^2\right| R_1,R_1^{(n)},R_{-1}^{(n)}\right]}+\nu_2\right]\\
=&\left(\sup_{w\in D}\frac{\|D^2g(w)\|}{1+\|w\|}\right)\left[1+2\sqrt{R_{-1}^{(n)}(1)}+\nu_2\right]\\
\leq &\left(1+2\sqrt{\frac{1}{8}+\frac{\nu_1}{n}}+\nu_2\right)\left(\sup_{w\in D}\frac{\|D^2g(w)\|}{1+\|w\|}\right)\\
&\text{C)}\quad \mathbbm{E}\left[\left.\sup_{w\in D}\frac{\left\|Dg\left(w-Y_n^{-1}+I_n\right)\right\|}{1+\|w\|^2}\right|R_1,R_1^{(n)},R_{-1}^{(n)}\right]\\
\leq &\mathbbm{E}\left[\left.\sup_{w\in D}\frac{\left\|Dg\left(w-Y_n^{-1}+I_n\right)\right\|}{1+\left\|w-Y_n^{-1}+I_n\right\|^2}\cdot \frac{1+\left\|w-Y_n^{-1}+I_n\right\|^2}{1+\|w\|^2}\right|R_1,R_1^{(n)},R_{-1}^{(n)}\right]\\
\stackrel{Doob}\leq& 3\left(\sup_{w\in D}\frac{\|Dg(w)\|}{1+\|w\|^2}\right)\left[1+4\mathbbm{E}\left[\left.\left(Y_n^{-1}(1)\right)^2\right|R_1,R_1^{(n)},R_{-1}^{(n)}\right]+\nu_2^2\right]\\
=&3\left(\sup_{w\in D}\frac{\|Dg(w)\|}{1+\|w\|^2}\right)\left[1+4R_{-1}^{n}(1)+\nu_2^2\right]\\
\leq &\left(\frac{9}{2}+\frac{12\nu_1}{n}+3\nu_2^2\right)\left(\sup_{w\in D}\frac{\|Dg(w)\|}{1+\|w\|^2}\right)\\
&\text{D)}\quad \mathbbm{E}\left[\left.\sup_{w\in D}\frac{\left|g\left(w-Y_n^{-1}\right)\right|}{1+\|w\|^3}\right|R_1,R_1^{(n)}, R_{-1}^{(n)}\right]\\
\leq&\mathbbm{E}\left[\left.\sup_{w\in D}\frac{\left|g\left(w-Y_n^{-1}+I_n\right)\right|}{1+\left\|w-Y_n^{-1}+I_n\right\|^3}\cdot \frac{1+\left\|w-Y_n^{-1}+I_n\right\|^3}{1+\|w\|^3}\right|R_1,R_1^{(n)},R_{-1}^{(n)}\right]\\
\stackrel{Doob}\leq& 9\left(\sup_{w\in D}\frac{|g(w)|}{1+\|w\|^3}\right)\left[1+\frac{27}{8}\mathbbm{E}\left[\left.\left|Y_n^{-1}(1)\right|^3\right|R_1,R_1^{(n)},R_{-1}^{(n)}\right]+\nu_2^3\right]\\
\leq &9\left(\sup_{w\in D}\frac{|g(w)|}{1+\|w\|^3}\right)\left[1+\frac{27}{8}\left(3\left(R_{-1}^{(n)}(1)\right)^2+n^{-1}R_{-1}^{(n)}(1)\right)^{3/4}+\nu_2^3\right]\\
\leq &\left(9+\frac{243}{8}\cdot \left(3\left(\frac{1}{8}+\frac{\nu_1}{n}\right)^2+\frac{1}{8n}+\frac{\nu_1}{n^2}\right)^{3/4}+9\nu_2^3\right)\left(\sup_{w\in D}\frac{|g(w)|}{1+\|w\|^3}\right).
\end{align*}
Therefore:
\begin{align}
&\mathbbm{E}\left[\|g^{(1)}\|_M|R_1,R_1^{(n)},R_{-1}^{(n)}\right]\nonumber\\
\leq& \left[9+2\sqrt{\frac{1}{8}+\frac{\nu_1}{n}}+\frac{12\nu_1}{n}+\frac{243}{8}\cdot \left(\frac{3}{64}+\frac{6\nu_1+1}{8n}+\frac{\nu_1+\nu_1^2}{n^2}\right)^{3/4}+\nu_2+3\nu_2^2+9\nu_2^3\right]\|g\|_M\label{11.33}
\end{align}
Similarly:
\begin{align*}
&\text{A)}\quad\mathbbm{E}\left[\left.\sup_{w,h\in D}\frac{\left\|D^2g\left(Z_1-(w+h)+I_n\right)-D^2g\left(Z_1-w+I_n\right)\right\|}{\|h\|}\right|R_1,R_1^{(n)},R_{-1}^{(n)}\right]\\
\leq&\sup_{x,y\in D}\frac{\|D^2g(x+y)-D^2g(x)\|}{\|y\|}\\
&\text{B)}\quad \mathbbm{E}\left[\left.\sup_{w\in D}\frac{\left\|D^2g\left(Z_1-w+I_n\right)\right\|}{1+\|w\|}\right|R_1,R_1^{(n)},R_{-1}^{(n)}\right]\\
\leq&\mathbbm{E}\left[\left.\sup_{w\in D}\frac{\left\|D^2g\left(Z_1-w+I_n\right)\right\|}{1+\left\|Z_1-w+I_n\right\|}\cdot \frac{1+\left\|Z_1-w+I_n\right\|}{1+\|w\|}\right|R_1,R_1^{(n)},R_{-1}^{(n)}\right]\\
\stackrel{Doob}\leq &\left(\sup_{w\in D}\frac{\|D^2g(w)\|}{1+\|w\|}\right)\left[1+2\sqrt{\mathbbm{E}\left[\left.\left(Z_1(1)\right)^2\right| R_1,R_1^{(n)},R_{-1}^{(n)}\right]}+\nu_2\right]\\
=&\left(\sup_{w\in D}\frac{\|D^2g(w)\|}{1+\|w\|}\right)\left[1+2\sqrt{R_1(1)}+\nu_2\right]\\
\leq& \left(1+\frac{\sqrt{2}}{2}+\nu_2\right)\left(\sup_{w\in D}\frac{\|D^2g(w)\|}{1+\|w\|}\right)\\
&\text{C)}\quad \mathbbm{E}\left[\left.\sup_{w\in D}\frac{\left\|Dg\left(Z_1-w+I_n\right)\right\|}{1+\|w\|^2}\right|R_1,R_1^{(n)},R_{-1}^{(n)}\right]\\
\leq& \mathbbm{E}\left[\left.\sup_{w\in D}\frac{\left\|Dg\left(Z_1-w+I_n\right)\right\|}{1+\left\|Z_1-w+I_n\right\|^2}\cdot \frac{1+\left\|Z_1-w+I_n\right\|^2}{1+\|w\|^2}\right|R_1,R_1^{(n)},R_{-1}^{(n)}\right]\\
\stackrel{Doob}\leq &3\left(\sup_{w\in D}\frac{\|Dg(w)\|}{1+\|w\|^2}\right)\left[1+4\mathbbm{E}\left[\left.\left(Z_1(1)\right)^2\right|R_1,R_1^{(n)},R_{-1}^{(n)}\right]+\nu_2^2\right]\\
=&3\left(\sup_{w\in D}\frac{\|Dg(w)\|}{1+\|w\|^2}\right)\left[1+4R_1(1)+\nu_2^2\right]\leq \left(\frac{9}{2}+3\nu_2^2\right)\left(\sup_{w\in D}\frac{\|Dg(w)\|}{1+\|w\|^2}\right)\\
&\text{D)}\quad \mathbbm{E}\left[\left.\sup_{w\in D}\frac{\left|g\left(Z_1-w+I_n\right)\right|}{1+\|w\|^3}\right|R_1,R_1^{(n)}, R_{-1}^{(n)}\right]\\
\leq&\mathbbm{E}\left[\left.\sup_{w\in D}\frac{\left|g\left(Z_1-w+I_n\right)\right|}{1+\left\|Z_1-w+I_n\right\|^3}\cdot \frac{1+\left\|Z_1-w+I_n\right\|^3}{1+\|w\|^3}\right|R_1,R_1^{(n)},R_{-1}^{(n)}\right]\\
\stackrel{Doob}\leq &9\left(\sup_{w\in D}\frac{|g(w)|}{1+\|w\|^3}\right)\left[1+\frac{27}{8}\mathbbm{E}\left[\left.\left|Z_1(1)\right|^3\right|R_1,R_1^{(n)},R_{-1}^{(n)}\right]+\nu_2^3\right]\\
\leq& 9\left(\sup_{w\in D}\frac{|g(w)|}{1+\|w\|^3}\right)\left[1+\frac{27\sqrt{2}}{4\sqrt{\pi}}R_1(1)^{3/2}+\nu_2^3\right]\\
\leq& \left(9+\frac{243}{64\sqrt{\pi}}+9\nu_2^3\right)\left(\sup_{w\in D}\frac{|g(w)|}{1+\|w\|^3}\right).
\end{align*}
Therefore:
\begin{equation}\label{11.34}
\mathbbm{E}\left[\|g^{(-1)}\|_M|R_1,R_1^{(n)},R_{-1}^{(n)}\right]\leq \left(9+\frac{243}{64\sqrt{\pi}}+\nu_2+3\nu_2^2+9\nu_2^3\right)\|g\|_M.
\end{equation}
Also:
\begin{align*}
&\text{A)}\quad\mathbbm{E}\left[\left.\sup_{w,h\in D}\frac{\left\|D^2g\left(Z_1-Z_{-1}+(w+h)\right)-D^2g\left(Z_1-Z_{-1}+w\right)\right\|}{\|h\|}\right|R_1,R_1^{(n)},R_{-1}^{(n)}\right]\\
\leq&\sup_{x,y\in D}\frac{\|D^2g(x+y)-D^2g(x)\|}{\|y\|}\\
&\text{B)}\quad \mathbbm{E}\left[\left.\sup_{w\in D}\frac{\left\|D^2g\left(Z_1-Z_{-1}+w\right)\right\|}{1+\|w\|}\right|R_1,R_1^{(n)},R_{-1}^{(n)}\right]\\
\leq&\mathbbm{E}\left[\left.\sup_{w\in D}\frac{\left\|D^2g\left(Z_1-Z_{-1}+w\right)\right\|}{1+\left\|Z_1-Z_{-1}+w\right\|}\cdot \frac{1+\left\|Z_1-Z_{-1}+w\right\|}{1+\|w\|}\right|R_1,R_1^{(n)},R_{-1}^{(n)}\right]\\
\stackrel{Doob}\leq &\left(\sup_{w\in D}\frac{\|D^2g(w)\|}{1+\|w\|}\right)\left[1+2\sqrt{\mathbbm{E}\left[\left.\left(Z_1(1)-Z_{-1}(1)\right)^2\right| R_1,R_1^{(n)},R_{-1}^{(n)}\right]}\right]\\
=&\left(\sup_{w\in D}\frac{\|D^2g(w)\|}{1+\|w\|}\right)\left[1+2\sqrt{2R_1(1)}\right]\\
\leq &2\left(\sup_{w\in D}\frac{\|D^2g(w)\|}{1+\|w\|}\right)\\
&\text{C)}\quad \mathbbm{E}\left[\left.\sup_{w\in D}\frac{\left\|Dg\left(Z_1-Z_{-1}+w\right)\right\|}{1+\|w\|^2}\right|R_1,R_1^{(n)},R_{-1}^{(n)}\right]\\
\leq &\mathbbm{E}\left[\left.\sup_{w\in D}\frac{\left\|Dg\left(Z_1-Z_{-1}+w\right)\right\|}{1+\left\|Z_1-Z_{-1}+w\right\|^2}\cdot \frac{1+\left\|Z_1-Z_{-1}+w\right\|^2}{1+\|w\|^2}\right|R_1,R_1^{(n)},R_{-1}^{(n)}\right]\\
\stackrel{Doob}\leq &2\left(\sup_{w\in D}\frac{\|Dg(w)\|}{1+\|w\|^2}\right)\left[1+4\mathbbm{E}\left[\left.\left(Z_1(1)-Z_{-1}(1)\right)^2\right|R_1,R_1^{(n)},R_{-1}^{(n)}\right]\right]\\
=&2\left(\sup_{w\in D}\frac{\|Dg(w)\|}{1+\|w\|^2}\right)\left[1+8R_1(1)\right]\\
\leq &4\left(\sup_{w\in D}\frac{\|Dg(w)\|}{1+\|w\|^2}\right)\\
&\text{D)}\quad \mathbbm{E}\left[\left.\sup_{w\in D}\frac{\left|g\left(Z_1-Z_{-1}+w\right)\right|}{1+\|w\|^3}\right|R_1,R_1^{(n)}, R_{-1}^{(n)}\right]\\
\leq&\mathbbm{E}\left[\left.\sup_{w\in D}\frac{\left|g\left(Z_1-Z_{-1}+w\right)\right|}{1+\left\|Z_1-Z_{-1}+w\right\|^3}\cdot \frac{1+\left\|Z_1-Z_{-1}+w\right\|^3}{1+\|w\|^3}\right|R_1,R_1^{(n)},R_{-1}^{(n)}\right]\\
\stackrel{Doob}\leq &4 \left(\sup_{w\in D}\frac{|g(w)|}{1+\|w\|^3}\right)\left[1+\frac{27}{8}\mathbbm{E}\left[\left.\left|Z_1(1)-Z_{-1}(1)\right|^3\right|R_1,R_1^{(n)},R_{-1}^{(n)}\right]\right]\\
\leq &4\left(\sup_{w\in D}\frac{|g(w)|}{1+\|w\|^3}\right)\left[1+\frac{27\sqrt{2}}{4\sqrt{\pi}}(2R_1(1))^{3/2}\right]\\
\leq &\left(4+\frac{27\sqrt{2}}{8\sqrt{\pi}}\right)\left(\sup_{w\in D}\frac{|g(w)|}{1+\|w\|^3}\right).
\end{align*}
So:
\begin{equation}\label{11.35}
\mathbbm{E}\left[\|g^{(0)}\|_M|R_1,R_1^{(n)},R_{-1}^{(n)}\right]\leq \left(4+\frac{27\sqrt{2}}{8\sqrt{\pi}}\right)\|g\|_M.
\end{equation}

Now, the Moran model and the Wright-Fisher diffusion can be coupled using the Donnelly-Kurtz look-down construction (see the discussion below Theorem \ref{theorem_moran}). In this construction first the Wright-Fisher diffusion $M$ is realised and then the Moran model $M_n$ is constructed by describing $nM_n(s)$ as a $\textbf{Binomial}(n,M(s))$ random variable. Note that:
\begin{align}
&\text{A)}\quad\mathbbm{E}\sqrt{\|R_1-R_1^{(n)}\|}\nonumber\\
\leq& \mathbbm{E}\sqrt{\int_0^1 \left|\frac{1}{2}X(s)(1-X(s))-\left(\frac{1}{2}X_n(s)+\frac{\nu_2}{n}\right)\left(1-X_n(s)\right)\right|ds}\nonumber\\
\stackrel{Jensen}\leq& \sqrt{\int_0^1 \mathbbm{E}\left|\frac{1}{2}X(s)(1-X(s))-\left(\frac{1}{2}X_n(s)+\frac{\nu_2}{n}\right)\left(1-X_n(s)\right)\right|ds}\nonumber\\
\leq& \left\lbrace \int_0^1 \left(\frac{1}{2}\mathbbm{E}\left[\mathbbm{E}[|X_n(s)-X(s)||X(s)]\right]+\frac{1}{2}\mathbbm{E}\left[\mathbbm{E}[|X_n^2(s)-X^2(s)||X(s)]\right]\right.\right.\nonumber\\
&\left.\left.+\frac{\nu_2}{n}\mathbbm{E}\left[\mathbbm{E}[|1-X_n(s)||X(s)]\right]\right)ds\right\rbrace^{1/2}\nonumber\\
\leq& \left\lbrace\int_0^1 \left(\frac{1}{2}\mathbbm{E}\left[\sqrt{\text{Var}\left[\left.X_n(s)\right|X(s)\right]}\right]+\frac{1}{2}\mathbbm{E}\left[\sqrt{\text{Var}\left[\left.X^2_n(s)\right|X(s)\right]}\right]\right.\right.\nonumber\\
&\left.\left.+\frac{1}{2}\mathbbm{E}\left[\frac{X(s)(1-X(s))}{n}\right]+\frac{\nu_2}{n}\mathbbm{E}[1-X(s)]\right)ds\right\rbrace^{1/2}\nonumber\\
=&\left\lbrace\int_0^1 \left(\frac{1}{2}\mathbbm{E}\left[\sqrt{\frac{X(s)(1-X(s))}{n}}\right]+\frac{1}{2}\mathbbm{E}\left[\left(\frac{X(s)(1-7X(s)+6nX(s))}{n^3}\right.\right.\right.\right.\nonumber\\
&+\frac{X(s)(12X^2(s)-20nX^2(s)+4n^2X^2(s)-6X^3(s)+10nX^3(s)}{n^3}\nonumber\\
&\left.\left.\left.\left.-\frac{8X^4(s)}{n}\right)^{1/2}\right]+\mathbbm{E}\left[\frac{(X(s)+2\nu_2)(1-X(s))}{2n}\right]\right)ds\right\rbrace^{1/2}\nonumber\\
\leq& \sqrt{\frac{1}{4}n^{-1/2}+\frac{1}{2}\sqrt{13n^{-3}+16n^{-2}+4n^{-1}}+\frac{1+2\nu_2}{2}n^{-1}}\nonumber\\
&\text{B)}\quad \mathbbm{E}\|R_1-R_1^{(n)}\|^{3/2}\nonumber\\
\leq&\mathbbm{E}\left[\left(\int_0^1 \left|\frac{1}{2}X(s)(1-X(s))-\left(\frac{1}{2}X_n(s)+\frac{\nu_2}{n}\right)(1-X_n(s))\right|ds\right)^{3/2}\right]\nonumber\\
\stackrel{\text{H\"older}}\leq&\mathbbm{E}\left[\int_0^1 \left|\frac{1}{2}X(s)(1-X(s))-\left(\frac{1}{2}X_n(s)+\frac{\nu_2}{n}\right)(1-X_n(s))\right|^{3/2}ds\right]\nonumber\\
\stackrel{\text{Jensen}}\leq&\sqrt{3}\int_0^1\left\lbrace\frac{\sqrt{2}}{4}\mathbbm{E}\left[\left.\mathbbm{E}\left|X_n(s)-X(s)\right|^{3/2}\right|X(s)\right]\right.\nonumber\\
&+\frac{\sqrt{2}}{4}\mathbbm{E}\left[\left.\mathbbm{E}\left|X_n^2(s)-X^2(s)\right|^{3/2}\right|X(s)\right]\nonumber\\
&\left.+\frac{\nu_2^{3/2}}{n^{3/2}}\mathbbm{E}\left[\mathbbm{E}\left[|1-X_n(s)|^{3/2}|X(s)\right]\right]\right\rbrace ds\nonumber\\
\stackrel{\text{Jensen}}\leq& \sqrt{3}\int_0^1\left\lbrace\frac{\sqrt{2}}{4}\mathbbm{E}\left[\left(\text{Var}\left[\left.X_n(s)\right|X(s)\right]\right)^{3/4}\right]\right.\nonumber\\
&+\frac{1}{2}\mathbbm{E}\left[\left.\mathbbm{E}\left|X_n^2(s)-\left(\frac{X(s)(1-X(s))}{n}+X^2(s)\right)\right|^{3/2}\right|X(s)\right]\nonumber\\
&+\frac{1}{2}\mathbbm{E}\left[\left(\frac{X(s)(1-X(s))}{n}\right)^{3/2}\right]\nonumber\\
&\left.+\frac{\nu_2^{3/2}}{n^{3/2}}\mathbbm{E}\left[\left(\frac{(1-X(s))X(s)}{n}+\left(1-X(s)\right)^2\right)^{3/4}\right]\right\rbrace ds\nonumber\\
\leq& \sqrt{3}\int_0^1\left\lbrace\frac{\sqrt{2}}{4}\mathbbm{E}\left[\left(\text{Var}\left[\left.X_n(s)\right|X(s)\right]\right)^{3/4}\right]+\left(\frac{1}{16}+\nu_2^{3/2}\right)n^{-3/2}\right.\nonumber\\
&\left.+\frac{1}{2}\mathbbm{E}\left[\left.\mathbbm{E}\left|X_n^2(s)-\left(\frac{X(s)(1-X(s))}{n}+X^2(s)\right)\right|^{3/2}\right|X(s)\right]\right\rbrace ds\nonumber\\
\stackrel{\text{Jensen}}\leq&\sqrt{3}\int_0^1\left\lbrace\frac{\sqrt{2}}{4}\mathbbm{E}\left[\left(\text{Var}\left[\left.X_n(s)\right|X(s)\right]\right)^{3/4}\right]\right.\nonumber\\
&\left.+\frac{1}{2}\mathbbm{E}\left[\left(\text{Var}\left[\left.X_n^2(s)\right|X(s)\right]\right)^{3/4}\right]+\left(\frac{1}{16}+\nu_2^{3/2}\right)n^{-3/2}\right\rbrace ds \nonumber\\
=&\sqrt{3}\int_0^1\left\lbrace\frac{\sqrt{2}}{4}\mathbbm{E}\left[\left(\frac{X(s)(1-X(s))}{n}\right)^{3/4}\right]+\frac{1}{2}\mathbbm{E}\left[\left(\frac{X(s)(1-7X(s))}{n^3}\right.\right.\right.\nonumber\\
&+\frac{X(s)(6nX(s)+12X^2(s)-20nX^2(s)+4n^2X^2(s)-6X^3(s))}{n^3}\nonumber\\
&\left.\left.\left.+\frac{X(s)(10nX^3(s)-8n^2X^3(s))}{n^3}\right)^{3/4}\right]+\left(\frac{1}{16}+\nu_2^{3/2}\right)n^{-3/2}\right\rbrace ds\nonumber\\
\leq& \sqrt{3}\left(\frac{\sqrt{2}}{32}n^{-3/4}+\frac{1}{2}\left(13n^{-3}+16n^{-2}+4n^{-1}\right)^{3/4}+\left(\frac{1}{16}+\nu_2^{3/2}\right)n^{-3/2}\right)\nonumber\\
&\text{C)}\quad\mathbbm{E}\sqrt{\|R_{-1}-R_{-1}^{(n)}\|}\nonumber\\
\leq& \mathbbm{E}\sqrt{\int_0^1 \left|\frac{1}{2}X(s)(1-X(s))-\left(\frac{1}{2}(1-X_n(s))+\frac{\nu_1}{n}\right)X_n(s)\right|ds}\nonumber\\
\stackrel{Jensen}\leq& \sqrt{\int_0^1 \mathbbm{E}\left|\frac{1}{2}X(s)(1-X(s))-\left(\frac{1}{2}(1-X_n(s))+\frac{\nu_1}{n}\right)X_n(s)\right|ds}\nonumber\\
\leq& \left\lbrace \int_0^1 \left(\frac{1}{2}\mathbbm{E}\left[\mathbbm{E}[|X_n(s)-X(s)||X(s)]\right]+\frac{1}{2}\mathbbm{E}\left[\mathbbm{E}[|X_n^2(s)-X^2(s)||X(s)]\right]\right.\right.\nonumber\\
&\left.\left.+\frac{\nu_1}{n}\mathbbm{E}\left[\mathbbm{E}[X_n(s)|X(s)]\right]\right)ds\right\rbrace^{1/2}\nonumber\\
\leq& \left\lbrace\int_0^1 \left(\frac{1}{2}\mathbbm{E}\left[\sqrt{\text{Var}\left[\left.X_n(s)\right|X(s)\right]}\right]+\frac{1}{2}\mathbbm{E}\left[\sqrt{\text{Var}\left[\left.X^2_n(s)\right|X(s)\right]}\right]\right.\right.\nonumber\\
&\left.\left.+\frac{1}{2}\mathbbm{E}\left[\frac{X(s)(1-X(s))}{n}\right]+\frac{\nu_1}{n}\mathbbm{E}[X(s)]\right)ds\right\rbrace^{1/2}\nonumber\\
=&\left\lbrace\int_0^1 \left(\frac{1}{2}\mathbbm{E}\left[\sqrt{\frac{X(s)(1-X(s))}{n}}\right]+\frac{1}{2}\mathbbm{E}\left[\left(\frac{X(s)(1-7X(s)+6nX(s))}{n^3}\right.\right.\right.\right.\nonumber\\
&+\frac{X(s)(12X^2(s)-20nX^2(s)+4n^2X^2(s)-6X^3(s)+10nX^3(s))}{n^3}\nonumber\\
&\left.\left.\left.\left.-\frac{8X^4(s)}{n}\right)^{1/2}\right]+\mathbbm{E}\left[\frac{X(s)(1-X(s)+2\nu_1)}{2n}\right]\right)ds\right\rbrace^{1/2}\nonumber\\
\leq& \sqrt{\frac{1}{4}n^{-1/2}+\frac{1}{2}\sqrt{13n^{-3}+16n^{-2}+4n^{-1}}+\frac{1+2\nu_1}{2}n^{-1}}\nonumber\\
&\text{D)}\quad \mathbbm{E}\|R_{-1}-R_{-1}^{(n)}\|^{3/2}\nonumber\\
\leq&\mathbbm{E}\left[\left(\int_0^1 \left|\frac{1}{2}X(s)(1-X(s))-\left(\frac{1}{2}(1-X_n(s))+\frac{\nu_1}{n}\right)X_n(s)\right|ds\right)^{3/2}\right]\nonumber\\
\stackrel{\text{H\"older}}\leq&\mathbbm{E}\left[\int_0^1 \left|\frac{1}{2}X(s)(1-X(s))-\left(\frac{1}{2}(1-X_n(s))+\frac{\nu_1}{n}\right)X_n(s)\right|^{3/2}ds\right]\nonumber\\
\stackrel{Jensen}\leq&\sqrt{3}\int_0^1\left\lbrace\frac{\sqrt{2}}{4}\mathbbm{E}\left[\left.\mathbbm{E}\left|X_n(s)-X(s)\right|^{3/2}\right|X(s)\right]\right.\nonumber\\
&+\frac{\sqrt{2}}{4}\mathbbm{E}\left[\left.\mathbbm{E}\left|X_n^2(s)-X^2(s)\right|^{3/2}\right|X(s)\right]\nonumber\\
&\left.+\frac{\nu_1^{3/2}}{n^{3/2}}\mathbbm{E}\left[\mathbbm{E}\left[|X_n(s)|^{3/2}|X(s)\right]\right]\right\rbrace ds\nonumber\\
\stackrel{\text{Jensen}}\leq& \sqrt{3}\int_0^1\left\lbrace\frac{\sqrt{2}}{4}\mathbbm{E}\left[\left(\text{Var}\left[\left.X_n(s)\right|X(s)\right]\right)^{3/4}\right]\right.\nonumber\\
&\left.+\frac{1}{2}\mathbbm{E}\left[\left.\mathbbm{E}\left|X_n^2(s)-\left(\frac{X(s)(1-X(s))}{n}+X^2(s)\right)\right|^{3/2}\right|X(s)\right]\right.\nonumber\\
&+\frac{1}{2}\mathbbm{E}\left[\left(\frac{X(s)(1-X(s))}{n}\right)^{3/2}\right]\nonumber\\
&\left.+\frac{\nu_1^{3/2}}{n^{3/2}}\mathbbm{E}\left[\left(\frac{(1-X(s))X(s)}{n}+X^2(s)\right)^{3/4}\right]\right\rbrace ds\nonumber\\
\leq& \sqrt{3}\int_0^1\left\lbrace\frac{\sqrt{2}}{4}\mathbbm{E}\left[\left(\text{Var}\left[\left.X_n(s)\right|X(s)\right]\right)^{3/4}\right]+\left(\frac{1}{16}+\nu_1^{3/2}\right)n^{-3/2}\right.\nonumber\\
&\left.+\frac{1}{2}\mathbbm{E}\left[\left.\mathbbm{E}\left|X_n^2(s)-\left(\frac{X(s)(1-X(s))}{n}+X^2(s)\right)\right|^{3/2}\right|X(s)\right]\right\rbrace ds\nonumber\\
\stackrel{\text{Jensen}}\leq&\sqrt{3}\int_0^1\left\lbrace\frac{\sqrt{2}}{4}\mathbbm{E}\left[\left(\text{Var}\left[\left.X_n(s)\right|X(s)\right]\right)^{3/4}\right]+\frac{1}{2}\mathbbm{E}\left[\left(\text{Var}\left[\left.X_n^2(s)\right|X(s)\right]\right)^{3/4}\right]\right.\nonumber\\
&+\left.\left(\frac{1}{16}+\nu_1^{3/2}\right)n^{-3/2}\right\rbrace ds \nonumber\\
=&\sqrt{3}\int_0^1\left\lbrace\frac{\sqrt{2}}{4}\mathbbm{E}\left[\left(\frac{X(s)(1-X(s))}{n}\right)^{3/4}\right]+\frac{1}{2}\mathbbm{E}\left[\left(\frac{X(s)(1-7X(s))}{n^3}\right.\right.\right.\nonumber\\
&+\frac{X(s)(6nX(s)+12X^2(s)-20nX^2(s)+4n^2X^2(s)-6X^3(s))}{n^3} \nonumber\\
&\left.\left.\left.+\frac{X(s)(10nX^3(s)-8n^2X^3(s))}{n^3}\right)^{3/4}\right]+\left(\frac{1}{16}+\nu_1^{3/2}\right)n^{-3/2}\right\rbrace ds\nonumber\\
\leq& \sqrt{3}\left(\frac{\sqrt{2}}{32}n^{-3/4}+\frac{1}{2}\left(13n^{-3}+16n^{-2}+4n^{-1}\right)^{3/4}+\left(\frac{1}{16}+\nu_1^{3/2}\right)n^{-3/2}\right)\label{11.36}.
\end{align}
Also:
\begin{align}
&\text{A)}\quad\mathbbm{E}\|I-I_n\|\leq \int_0^1(\nu_1+\nu_2)\mathbbm{E}\left[\mathbbm{E}\left[|X_n(s)-X(s)||X(s)\right]\right]ds\nonumber\\
\leq& \int_0^1(\nu_1+\nu_2)\mathbbm{E}\left[\sqrt{\text{Var}\left[X_n(s)|X(s)\right]}\right]ds\nonumber\\
\leq& \int_0^1(\nu_1+\nu_2)\mathbbm{E}\left[\sqrt{\frac{X(s)(1-X(s))}{n}}\right]ds\leq\frac{1}{2}(\nu_1+\nu_2)n^{-1/2}\nonumber\\
&\text{B)}\quad\mathbbm{E}\|I-I_n\|^3\stackrel{\text{H\"older}}\leq \int_0^1\mathbbm{E}\left|(\nu_1+\nu_2)\left(X_n(s)-X(s)\right)\right|^3ds\nonumber\\
\leq &(\nu_1+\nu_2)^3\int_0^1\left[\mathbbm{E}\left[\left(X_n(s)-X(s)\right)^4\right]\right]^{3/4}ds\nonumber\\
\leq& (\nu_1+\nu_2)^3\int_0^1 \mathbbm{E}\left[\left(\frac{X(s)(1-7X(s)+7nX(s)+12X^2(s)-18nX^2(s)}{n^4}\right.\right.\nonumber\\
&\left.\left.+\frac{X(s)(6n^2X^2(s)-6X^3(s)+11nX^3(s)-6n^2X^3(s)+n^3X^3(s))}{n^4}\right)^{3/4}\right]ds \nonumber\\
\leq& (\nu_1+\nu_2)^3 \left(13n^{-4}+18n^{-3}+6n^{-2}+n^{-1}\right)^{3/4}.
\label{11.37}
\end{align}
We now combine (\ref{11.29}),(\ref{11.31}), (\ref{11.32}), (\ref{11.33}), (\ref{11.34}), (\ref{11.35}), (\ref{11.36}), (\ref{11.37}) to obtain:
\begin{align*}
&|\mathbbm{E}g(M)-\mathbbm{E}g(M_n)|\\
\leq& \|g\|_M\left[9+2\sqrt{\frac{1}{8}+\frac{\nu_1}{n}}+\frac{12\nu_1}{n}+\frac{243}{8}\cdot \left(\frac{3}{64}+\frac{6\nu_1+1}{8n}+\frac{\nu_1+\nu_1^2}{n^2}\right)^{3/4}+\nu_2+3\nu_2^2+9\nu_2^3\right]\\
&\cdot\left\lbrace \left(2+\frac{27\sqrt{2}}{16\sqrt{\pi}}\right)\sqrt{\frac{1}{4}n^{-1/2}+\frac{1}{2}\sqrt{13n^{-3}+16n^{-2}+4n^{-1}}+\frac{1+2\nu_2}{2}n^{-1}}\right.\\
&+\frac{27\sqrt{2}}{2\sqrt{\pi}}\sqrt{3}\left(\frac{\sqrt{2}}{32}n^{-3/4}+\frac{1}{2}\left(13n^{-3}+16n^{-2}+4n^{-1}\right)^{3/4}+\left(\frac{1}{16}+\nu_2^{3/2}\right)n^{-3/2}\right)\\
&+ n^{-1}\left[\frac{30+\frac{27}{4}\cdot 5^{1/3}}{\sqrt{\pi\log 2}}\sqrt{\log (n^2/4)}+\left(1+\left(\frac{3}{2}\right)^3\sqrt{\frac{2}{\pi}}\left(\frac{1}{8}+\frac{\nu_2}{n}\right)^{3/2}\right)\right.\\
&\cdot\left(\frac{1}{8}+\frac{\nu_2}{n}\right)(1+2e^{-1})\left.+1+\frac{\log(1+2e^{-1})+4\log n}{\log\log(n^2+2)}\right]\\
&+n^{-2}\frac{9\sqrt{\frac{1}{8}+\frac{\nu_2}{n}}}{2}\left(1+\frac{3}{8}n^2+3\nu_2n\right)^{1/2}\left[4+\frac{16701+1024(\log n)^3}{(\log \log (n^2+3))^3}\right]^{1/3}\\
&+\left.n^{-3}\left[\frac{2160}{\sqrt{\pi}(\log 2)^{3/2}}(\log (n^2/4+\nu_2n))^{3/2}+8+\frac{33402+2048(\log n)^3}{(\log\log (n^2+3))^3}\right]\right\rbrace\\
&+\|g\|_M\left(9+\frac{243}{64\sqrt{\pi}}+\nu_2+3\nu_2^2+9\nu_2^3\right)\\
&\cdot\left\lbrace \left(2+\frac{27\sqrt{2}}{16\sqrt{\pi}}\right)\sqrt{\frac{1}{4}n^{-1/2}+\frac{1}{2}\sqrt{13n^{-3}+16n^{-2}+4n^{-1}}+\frac{1+2\nu_1}{2}n^{-1}}\right.\\
&+\frac{27\sqrt{2}}{2\sqrt{\pi}}\sqrt{3}\left(\frac{\sqrt{2}}{32}n^{-3/4}+\frac{1}{2}\left(13n^{-3}+16n^{-2}+4n^{-1}\right)^{3/4}+\left(\frac{1}{16}+\nu_1^{3/2}\right)n^{-3/2}\right)\\
&+ n^{-1}\left[\frac{30+\frac{27}{4}\cdot 5^{1/3}}{\sqrt{\pi\log 2}}\sqrt{\log (n^2/4)}+\left(1+\left(\frac{3}{2}\right)^3\sqrt{\frac{2}{\pi}}\left(\frac{1}{8}+\frac{\nu_1}{n}\right)^{3/2}\right)\right.\\
&\cdot\left(\frac{1}{8}+\frac{\nu_1}{n}\right)(1+2e^{-1})\left.+1+\frac{\log(1+2e^{-1})+4\log n}{\log\log(n^2+2)}\right]\nonumber\\
&+n^{-2}\frac{9\sqrt{\frac{1}{8}+\frac{\nu_1}{n}}}{2}\left(1+\frac{3}{8}n^2+3\nu_1n)\right)^{1/2}\left[4+\frac{16701+1024(\log n)^3}{(\log \log (n^2+3))^3}\right]^{1/3}\\
&\left.+n^{-3}\left[\frac{2160}{\sqrt{\pi}(\log 2)^{3/2}}(\log (n^2/4+\nu_1n))^{3/2}+8+\frac{33402+2048(\log n)^3}{(\log\log(n^2+3))^3}\right]\right\rbrace\\
&+\|g\|_M\left(4+\frac{27\sqrt{2}}{8\sqrt{\pi}}\right)\\
&\cdot\left[\frac{1}{2}(1+2\nu_2)(\nu_1+\nu_2)n^{-1/2}+2(\nu_1+\nu_2)^3 \left(13n^{-4}+18n^{-3}+6n^{-2}+n^{-1}\right)^{3/4}\right]\\
\leq& \|g\|_M\left\lbrace \left(18+\nu_1^{1/2}+47\nu_1^{3/4}+31\nu_1^{3/2}+\nu_2+3\nu_2^2+9\nu_2^3\right)\right.\\
&\cdot\left(1.02\cdot 10^6+425\nu_2^{1/2}+623\nu_2+39\nu_2^{3/2}+7\nu_2^{5/2}\right)\\
&+ \left(12+3\nu_2+3\nu_2^2+9\nu_2^3\right)\left(1.02\cdot 10^6+425\nu_1^{1/2}+623\nu_1+39\nu_1^{3/2}+7\nu_1^{5/2}\right)\\
&\left.+7\left(\frac{1}{2}(1+2\nu_2)(\nu_1+\nu_2)+31(\nu_1+\nu_2)^3\right)\right\rbrace n^{-1/4}\\
&+\|g\|_M2112\left[\left(18+\nu_1^{1/2}+47\nu_1^{3/4}+31\nu_1^{3/2}+\nu_2+3\nu_2^2+9\nu_2^3\right)\left(\log\left(n^2/4+\nu_2n\right)\right)^{3/2}\right.\\
&+\left.\left(12+3\nu_2+3\nu_2^2+9\nu_2^3\right)\left(\log\left(n^2/4+\nu_1n\right)\right)^{3/2}\right]n^{-3}.\qed
\end{align*}
\begin{remark}
The term of order $n^{-1/4}$ appearing in the bound obtained in Theorem \ref{theorem_moran} is unexpected. It comes from the comparison of the two time changes $R_1^{(n)}$ and $R_1$ applied at certain points in the proof to the Poisson process and to Brownian Motion respectively. 

One would expect $n^{-1}\sqrt{\log n^2}$ to be closer to the true speed of convergence. This is because our process $M_n$ can be expressed as a difference of two scaled Poisson processes with parameters of order $n^2$, which resemble scaled random walks stopped at the $n^2$-th step. Order $n^{-1}\sqrt{\log n^2}$ would therefore be in line with the results presented in \cite{diffusion} in the context of scaled random walks and indeed with Theorem \ref{time_change_theorem}.

We can therefore guess that our bound is not sharp, yet a method avoiding the comparison of the aforementioned time changes is needed to improve it.
\end{remark}
\begin{remark}
A strategy similar to the one used in the proof of Theorem \ref{theorem_moran} may be used to obtain bounds on the distance between other continuous-time Markov chains and diffusions. A key ingredient in the proof is, however, a way of coupling the two for any fixed time.
\end{remark}

\section*{Appendix: Proof of Proposition \ref{prop_m1}}\label{suppB}
Note that the proof of Proposition 3.1 of \cite{functional_combinatorial} readily applies in this case up to and excluding (3.4) and it suffices to prove that $\liminf_{n\to\infty}\mathbbm{P}[Y_n\in B]\geq \mathbbm{P}[Z\in B]$ for all sets $B$ of the form $B=\bigcap_{1\leq l\leq L}B_l$, where $B_l=\lbrace w\in D:\|w-s_l\|<\gamma_l\rbrace$, $s_l\in C\left([0,1],\mathbbm{R} \right)$ and $\gamma_l$ is such that $\mathbbm{P}[Z\in \partial B_l]=0$.

We will condition on the fact that the minimum holding time (interval of constancy of $Y_n$) is of length greater that $r_n=\lambda_n^{-3}$. It follows from Theorems 2.1 and 2.2 of Chapter 5 in \cite{devroye} that if we condition on the number of holding times being equal to $i$, their lengths are distributed uniformly over the simplex $A_i=\left\lbrace (x_1,...,x_i):x_j\geq 0, \sum_{j=1}^i x_j\leq 1\right\rbrace$. Note that the probability of the minimum of them being greater or equal to $r_n$ is $(1-ir_n)^i$ if $i\leq r_n$ and $0$ otherwise. This is because $\text{Vol}(A_i)=\frac{1}{i!}$ and $\text{Vol}\left(\left\lbrace (x_1,...,x_i):x_j\geq r_n, \sum_{j=1}^i x_j\leq 1\right\rbrace\right)=\frac{(1-ir_n)^i}{i!}$. Therefore:
\begin{align}
&\mathbbm{P}\left[\text{minimal waiting time}\geq r_n\right]\nonumber\\
=&\sum_{i=1}^\infty \mathbbm{P}\left[\text{minimal waiting time}\geq r_n|\# \text{waiting times}=i\right]\mathbbm{P}\left[\# \text{waiting times}=i\right]\nonumber\\
=&\sum_{i=1}^{\lfloor\lambda_n^3\rfloor}(1-i\lambda_n^{-3})^ie^{-\lambda_n}\frac{(\lambda_n)^{i-1}}{(i-1)!}\xrightarrow{n\to \infty} 1.\label{13}
\end{align}
To see this note the following:
\begin{align*}
\text{A)} \sum_{i=\left\lceil\lambda_n^{5/4}\right\rceil}^{\lfloor \lambda_n^3\rfloor}(1-i\lambda_n^{-3})^ie^{-\lambda_n}\frac{(\lambda_n)^{i-1}}{(i-1)!}
&\leq e^{-\lambda_n}\left(\lambda_n^3-\lambda_n^{5/4}\right)\left(1-\lambda_n^{-7/4}\right)^{\lambda_n^{5/4}}\frac{\lambda_n^{\lambda_n^{5/4}-1}}{\left(\left\lceil \lambda_n^{5/4}\right\rceil-1\right)!}\\
&\leq\frac{\lambda_n^2\left\lceil\lambda_n^{5/4}\right\rceil}{e^{\lambda_n}}\cdot \lambda_n^{-\frac{1}{8}\left\lceil \lambda_n^{5/4}\right\rceil+\frac{9}{8}\left\lceil \lambda_n^{9/8}\right\rceil}\xrightarrow{n\to\infty}0
\end{align*}
\begin{align*}
\text{B)}\sum_{i=1}^{\left\lceil\lambda_n^{5/4}\right\rceil-1}(1-i\lambda_n^{-3})^ie^{-\lambda_n}\frac{(\lambda_n)^{i-1}}{(i-1)!}
\geq (1-\lambda_n^{-7/4})^{\left\lceil \lambda_n^{5/4}\right\rceil}e^{-\lambda_n}\sum_{i=1}^{\left\lceil\lambda_n^{5/4}\right\rceil-1}\frac{(\lambda_n)^{i-1}}{(i-1)!}\xrightarrow{n\to\infty}1,
\end{align*}
where the convergence in B) holds since $(1-\lambda_n^{-7/4})^{\left\lceil \lambda_n^{5/4}\right\rceil}\to 1$, $e^{-\lambda_n}\sum_{i=1}^{\infty}\frac{(\lambda_n)^{i-1}}{(i-1)!}=1$ and:

\begin{align*}
&e^{-\lambda_n}\sum_{i=\left\lceil\lambda_n^{5/4}\right\rceil}^{\infty}\frac{(\lambda_n)^{i-1}}{(i-1)!}\leq e^{-\lambda_n}\frac{\lambda_n^{\left\lceil\lambda_n^{5/4}\right\rceil}}{\left\lceil\lambda_n^{5/4}\right\rceil!}\cdot\frac{\left\lceil\lambda_n^{5/4}\right\rceil+1}{\left\lceil\lambda_n^{5/4}\right\rceil+1-\lambda_n}\xrightarrow{n\to\infty}0
\end{align*}
for instance, by Proposition A.2.3(ii) of \cite{janson}.
Furthermore, note that for $g_{l,n}^*$ defined by (3.6) in \cite{functional_combinatorial}:

\begin{align}
 &\liminf_{n\to\infty}\mathbbm{E}\left[\prod_{l=1}^L g^*_{l,n}(Y_n)\right]\nonumber\\
=&\liminf_{n\to\infty}\mathbbm{E}\left[\left.\prod_{l=1}^L g^*_{l,n}(Y_n)\right|\text{minimal waiting time}\geq r_n\right]\mathbbm{P}\left[\text{minimal waiting time}\geq r_n\right]\nonumber\\
 &+\liminf_{n\to\infty}\mathbbm{E}\left[\left.\prod_{l=1}^L g^*_{l,n}(Y_n)\right|\text{minimal waiting time}< r_n\right]\mathbbm{P}\left[\text{minimal waiting time}< r_n\right]\nonumber\\
  =&\liminf_{n\to\infty}\mathbbm{E}\left[\left.\prod_{l=1}^L g^*_{l,n}(Y_n)\right|\text{minimal waiting time}\geq r_n\right]\mathbbm{P}\left[\text{minimal waiting time}\geq r_n\right]\label{300}\\
 \text{because:}\nonumber\\
0\leq& \liminf_{n\to\infty}\mathbbm{E}\left[\left.\prod_{l=1}^L g^*_{l,n}(Y_n)\right|\text{minimal waiting time}< r_n\right]\mathbbm{P}\left[\text{minimal waiting time}< r_n\right]\nonumber\\
\leq& \liminf_{n\to\infty}\mathbbm{P}\left[\text{minimal waiting time}< r_n\right]\stackrel{(\ref{13})}=0.\nonumber
\end{align}
Following the same steps as in \cite{functional_combinatorial}, we obtain:
\begin{align*}
&\liminf_{n\to\infty}\mathbbm{P}\left[Y_n\in B\right]\geq \liminf_{n\to\infty}\mathbbm{P}\left[Y_n\in B\text{ and minimal waiting time}\geq r_n\right]\\
&\geq \liminf_{n\to\infty}\mathbbm{E}\left[\left.\prod_{l=1}^L g^*_{l,n}(Y_n)\right|\text{minimal waiting time}\geq r_n\right]\mathbbm{P}\left[\text{minimal waiting time}\geq r_n\right]\\
&\stackrel{(\ref{assumption}),(\ref{300})}{\geq}
\liminf_{n\to\infty}\left\lbrace\mathbbm{E}\left[\prod_{l=1}^L g^*_{l,n}(Z_n)\right]-C\tau_n\left\|\prod_{l=1}^L g^*_{l,n}\right\|_{M^0}\right\rbrace\\
&\geq\liminf_{n\to\infty}\left\lbrace\mathbbm{E}\left[\prod_{l=1}^L g^*_{l,n}(Z_n)\right]-C''\tau_n{p_n}^2(\epsilon\gamma)^{-2}\eta_n^{-3}\right\rbrace\\
&\stackrel{\text{Fatou}}{\geq} \mathbbm{E}\left[\liminf_{n\to\infty}\prod_{l=1}^L g^*_{l,n}(Z_n)\right]\geq\mathbbm{P}\left[\bigcap_{1\leq l\leq L}(\|Z-s_l\|<\gamma_l(1-\theta))\right]
\text{.}
\end{align*}
\qed
\section*{Acknowledgements}
The author would like to thank Professor Gesine Reinert, Professor Alison Etheridge and Professor Andrew Barbour for many helpful discussions and Professor Gesine Reinert for constructive comments on the early versions of the paper. The author is also grateful to Dr Sebastian Vollmer for pointing out a mistake in an early version of the proof of Proposition \ref{prop2.10}.

\bibliographystyle{alpha}
\bibliography{Bibliography}

\end{document}